\documentclass[reqno]{amsart}
\usepackage{amssymb}
\usepackage{amsmath}
\usepackage{amsfonts}

\newtheorem{theorem}{Theorem}[section]
\theoremstyle{plain}

\newtheorem{definition}{Definition}[section]

\newtheorem{lemma}{Lemma}[section]

\numberwithin{equation}{section}

\begin{document}
\title[Strongly damped wave equation]{Strongly damped wave equation with
exponential nonlinearities}
\author{\ Azer Khanmamedov\ }
\address{{\small Department of Mathematics,} {\small Faculty of Science,
Hacettepe University, Beytepe 06800}, {\small Ankara, Turkey}}
\email{azer@hacettepe.edu.tr; azer\_khan@yahoo.com }
\subjclass[2000]{ 35L05, 35B41}
\keywords{wave equation, global attractor}

\begin{abstract}
In this paper, we study the initial boundary value problem for two
dimensional strongly damped wave equation with exponentially growing source
and damping terms. We first show the well-posedness of this problem and then
prove the existence of the global attractor in $(H_{0}^{1}(\Omega )\cap
L^{\infty }(\Omega ))\times L^{2}(\Omega )$.
\end{abstract}

\maketitle


\section{Introduction}

The paper is devoted to the study of the strongly damped wave equation%
\begin{equation}
w_{tt}-\Delta w_{t}+f(w_{t})-\Delta w+g(w)=h\text{.}  \tag{1.1}
\end{equation}%
The semilinear strongly damped wave equations are quite interesting from a
physical viewpoint. For example, they arise in the modeling of the flow of
viscoelastic fluids (see [1, 2]) as well as in the theory of heat conduction
(see [3, 4]). One of the most important problems regarding these equations
is to analyse their long-time dynamics in terms of attractors. The
attractors for such equations have intensively been studied by many authors
under different types of hypotheses. We refer to [5-7] and the references
therein for strongly damped wave equations with the linear damping and
subcritical source term. In the critical source term case, the existence of
the attractors for strongly damped wave equations with the linear damping
was proved in \cite{8} and later in \cite{9}. The regularity of the
attractor, established in [8, 9], was proved in \cite{10}, for the critical
source term case. Later in \cite{11}, it was shown that the attractor of the
strongly damped wave equation with the critical source term, indeed,
attracts every bounded subset of $H_{0}^{1}(\Omega )\times L^{2}(\Omega )$
in the norm of $H_{0}^{1}(\Omega )\times H_{0}^{1}(\Omega )$. In \cite{12},
the authors proved the existence and regularity of the uniform attractor for
the nonautonomous strongly damped wave with the critical source term. The
attractors for the strongly damped wave equations with the source term like
polynomial of arbitrary degree were investigated in \cite{13}. In the
nonlinear subcritical damping term case, the attractors for the strongly
damped wave equations were studied in \cite{14} and \cite{15}. In \cite{16},
the authors investigated the attractors of the abstract second order
evolution equation with the damping term depending both on displacement and
velocity. In particular, the results obtained in \cite{16} can be applied to
the strongly damped wave equation with subcritical nonlinearities.
Attractors for\ strongly damped wave equations with the critical
displacement dependent damping and source terms were established in \cite{17}%
. Recently, in \cite{18} the authors have proved the existence of the
attractors for the equation (1.1), when the source term $g$ is subcritical
and nonmonotone damping term $f$ is critical. Later in \cite{19}, they have
improved this result for the case when both of $f$ and $g$ are critical and $%
\underset{x\in R}{\inf }$ $f^{\prime }(x)>-\lambda _{1}$, where $\lambda
_{1} $ is the first eigenvalue of Laplace operator.

The goal of this paper is to study the two dimensional equation (1.1) with
exponentially growing damping and source terms, in the space $%
(H_{0}^{1}(\Omega )\cap L^{\infty }(\Omega ))\times L^{2}(\Omega )$ instead
of the usual phase space\ $H_{0}^{1}(\Omega )\times L^{2}(\Omega )$. Hence,
in comparison with the papers mentioned above, we additionally need the $%
L^{\infty }(\Omega )-$estimate for the weak solutions. This estimate is also
important for the uniqueness and continuous dependence on initial data, in
the case when there is no growth condition on the derivative of the source
term. To achieve $L^{\infty }$ regularity of the weak solutions, we reduce
the strongly damped wave equation to the heat equation and use the
regularity property of the latter (see Lemma 3.2-3.3).

The paper is organized as follows. In the next section, we state the problem
and the main results. In section 3, we first prove the existence of the weak
solution and then establish its $L^{\infty }$ regularity, which plays a key
role for the uniqueness of the weak solution. After proving the uniqueness,
we show $L^{2}$ regularity for $w_{tt}$ and then continuous dependence of
the weak solution on initial data. In section 4, we first establish the
dissipativity, in particular the global boundedness of solutions in $%
L^{\infty }(\Omega )$ uniformly with respect to the initial data from a
bounded subset of $(H_{0}^{1}(\Omega )\cap L^{\infty }(\Omega ))\times
L^{2}(\Omega )$, and then prove asymptotic compactness which together with
the existence of the strict Lyapunov function lead to the existence of the
global attractor. Finally, in the last section, we give some auxiliary
lemmas.

\section{Statement of the problem and results}

We consider the following initial-boundary value problem:%
\begin{equation}
\left\{
\begin{array}{c}
w_{tt}-\Delta w_{t}+f(w_{t})-\Delta w+g(w)=h(x)\text{\ \ \ \ in \ }(0,\infty
)\times \Omega ,\text{ \ } \\
w=0\text{ \ \ \ \ \ \ \ \ \ \ \ \ \ \ \ \ \ \ \ \ \ \ \ \ \ \ \ \ \ \ \ \ \
\ \ \ \ \ \ \ \ \ \ \ \ \ \ on \ }(0,\infty )\times \partial \Omega , \\
w(0,\cdot )=w_{0}\text{ },\text{ \ \ \ \ \ \ \ \ \ \ }w_{t}(0,\cdot )=w_{1}%
\text{\ \ \ \ \ \ \ \ \ \ in \ }\Omega ,\text{ \ \ \ \ \ \ \ \ \ \ \ \ \ }%
\end{array}%
\right.  \tag{2.1}
\end{equation}%
where $\Omega \subset R^{2}$ is a bounded domain with smooth boundary, $h\in
L^{2}(\Omega )$ and the nonlinear functions $f,$ $g$ satisfy the following
conditions:%
\begin{equation}
\bullet \text{ }f\in C^{1}(R)\text{, \ }f(0)=0\text{,\ }\underset{s\in R}{%
\inf }f^{\prime }(s)>-\lambda _{{\small 1}}\text{, \ \ \ \ \ \ \ \ \ \ \ \ \
\ \ \ \ \ \ \ \ \ \ \ \ \ \ \ \ \ \ \ \ \ \ \ \ \ \ \ \ \ \ \ \ \ \ \ \ \ \
\ \ \ \ \ \ }  \tag{2.2}
\end{equation}%
\begin{equation}
\bullet \text{ }g\in C^{1}(R)\text{, }\underset{\left\vert s\right\vert
\rightarrow \infty }{\liminf }\frac{g(s)}{s}>-\lambda _{{\small 1}}\text{, \
}\left\vert g(s)\right\vert \leq c(1+e^{\left\vert s\right\vert ^{\gamma }})%
\text{, \ \ }\forall s\in R\text{, \ \ for some }\gamma \in \lbrack 1,2)%
\text{,}  \tag{2.3}
\end{equation}%
\begin{equation}
\bullet \text{ }\int\limits_{-\infty }^{\infty }\frac{\left\vert f^{\prime
}(s)\right\vert }{s\left( f(s)+\lambda _{{\small 1}}s\right) +1}ds<\infty
\text{,\ \ \ \ \ }\left\vert f(-s)\right\vert \leq c(1+\left\vert
f(s)\right\vert )\text{, \ \ \ }\forall s\in R\text{,\ \ \ \ \ \ \ \ \ \ \ \
\ \ \ \ \ \ \ }  \tag{2.4}
\end{equation}%
where $\lambda _{{\small 1}}=\underset{\varphi \in H_{0}^{1}(\Omega
),\varphi \neq 0}{\inf }\frac{\left\Vert \nabla \varphi \right\Vert
_{L^{2}(\Omega )}^{2}}{\left\Vert \varphi \right\Vert _{L^{2}(\Omega )}^{2}}$%
.\

It is easy to verify that for $\alpha \in \lbrack 0,1)$ the functions $%
f(s)=se^{\left\vert s\right\vert ^{\alpha }}$ satisfies the conditions (2.2)
and (2.4).

\begin{definition}
The function $w\in C([0,T];H_{0}^{1}(\Omega ))$ satisfying $w_{t}\in
C_{s}(0,T;L^{2}(\Omega ))\cap $\newline
$L^{2}(0,T;H_{0}^{1}(\Omega ))$, $f(w_{t})\in L^{1}(0,T;L^{1}(\Omega ))$, $%
g(w)\in L^{1}(0,T;L^{1}(\Omega ))$, $w(0,x)=w_{0}(x)$, $w_{t}(0,x)=w_{1}(x)$
and the equation%
\begin{equation*}
\frac{d}{dt}\left\langle w_{t},v\right\rangle +\left\langle \nabla
w_{t},v\right\rangle +\left\langle \nabla w,v\right\rangle +\left\langle
f(w_{t}),v\right\rangle +\left\langle g(w),v\right\rangle =\left\langle
h,v\right\rangle
\end{equation*}%
in the sense of distributions on $(0,T)$, for all $v\in H_{0}^{1}(\Omega
)\cap L^{\infty }(\Omega )$, is called the weak solution to the problem
(2.1) on $\left[ 0,T\right] \times \Omega $, where $C_{s}(0,T;L^{2}(\Omega
))=\left\{ u:u\in L^{\infty }(0,T;L^{2}(\Omega )),\left\langle u,\varphi
\right\rangle \in C[0,T]\text{, for}\right. $\newline
$\left. \text{every }\varphi \in L^{2}(\Omega )\right\} $ and $\left\langle
\psi ,\varphi \right\rangle =$ $\int\limits_{\Omega }\psi (x)\varphi (x)dx$.
\end{definition}

Our first result is the following well-posedness theorem:

\begin{theorem}
Assume that the conditions (2.2)-(2.4) are satisfied. Then for every $T>0$
and $(w_{0},w_{1})\in (H_{0}^{1}(\Omega )\cap L^{\infty }(\Omega ))\times
L^{2}(\Omega )$, the problem (2.1) admits a unique weak solution which
satisfies
\begin{equation*}
w\in C([0,T];H_{0}^{1}(\Omega )\cap L^{\infty }(\Omega )),\text{ }w_{t}\in
C([0,T];L^{2}(\Omega ))\cap L^{2}(0,T;H_{0}^{1}(\Omega )),\text{ }w_{tt}\in
L_{loc}^{2}(0,T;L^{2}(\Omega ))
\end{equation*}%
and the inequalities%
\begin{equation}
\left\{
\begin{array}{c}
\left\Vert w(t)\right\Vert _{H_{0}^{1}(\Omega )}+\left\Vert
w_{t}(t)\right\Vert _{L^{2}(\Omega )}+\underset{0}{\overset{t}{\int }}%
\left\Vert \nabla w_{t}(\tau )\right\Vert _{L^{2}(\Omega )}^{2}d\tau \\
+\underset{0}{\overset{t}{\int }}\left\langle f(w_{t}(\tau )),w_{t}(\tau
)\right\rangle d\tau \leq c_{1}(r(w_{0},w_{1}))\text{, \ \ \ \ \ \ \ \ } \\
\left\Vert w(t)\right\Vert _{L^{\infty }(\Omega )}\leq c_{2}(T,\text{ }%
r(w_{0},w_{1}))\text{ \ \ \ \ \ \ \ \ \ \ \ \ \ \ \ \ \ \ \ \ }%
\end{array}%
,\right. \text{ }\forall t\in \lbrack 0,T]\text{.}  \tag{2.5}
\end{equation}%
Moreover, if $\ v\in $ $C([0,T];H_{0}^{1}(\Omega )\cap L^{\infty }(\Omega
))\cap C^{1}([0,T];L^{2}(\Omega ))\cap W^{1,2}(0,T;H_{0}^{1}(\Omega ))\cap
W_{loc}^{2,2}(0,T;L^{2}(\Omega ))$ is also a weak solution to (2.1) with
initial data $(v_{0},v_{1})\in (H_{0}^{1}(\Omega )\cap L^{\infty }(\Omega
))\times L^{2}(\Omega )$, then
\begin{equation*}
\left\Vert w(t)-v(t)\right\Vert _{H_{0}^{1}(\Omega )}+\left\Vert
w_{t}(t)-v_{t}(t)\right\Vert _{L^{2}(\Omega )}\leq
\end{equation*}%
\begin{equation}
\leq c_{3}(T,\widetilde{r})\left( \left\Vert w_{0}-v_{0}\right\Vert
_{H_{0}^{1}(\Omega )}+\left\Vert w_{1}-v_{1}\right\Vert _{L^{2}(\Omega
)}\right) \text{, \ }\forall t\in \lbrack 0,T]\text{,}  \tag{2.6}
\end{equation}%
where $c_{1}:R_{+}\rightarrow R_{+}$ and $c_{i}:R_{+}\times R_{+}\rightarrow
R_{+}$ ($i=2,$ $3$) are nondecreasing functions with respect to each
variable, $r(w_{0},w_{1})=\left\Vert (w_{0},w_{1})\right\Vert
_{(H_{0}^{1}(\Omega )\cap L^{\infty }(\Omega ))\times L^{2}(\Omega )}$ and $%
\widetilde{r}=\max \left\{ r(w_{0},w_{1}),r(v_{0},v_{1})\right\} $.
\end{theorem}

Hence, by Theorem 2.1, the solution operator $%
S(t)(w_{0},w_{1})=(w(t),w_{t}(t))$ of the problem (2.1) generates a weakly
continuous (in the sense, if $\varphi _{n}\rightarrow \varphi $ strongly
then $S(t)\varphi _{n}\rightarrow S(t)\varphi $ weakly star)$\ $semigroup in
$(H_{0}^{1}(\Omega )\cap L^{\infty }(\Omega ))\times L^{2}(\Omega )$.

\begin{definition}
(\cite{20})\textit{\ Let }$\{V(t)\}_{{\small t\geq 0}}$\textit{\ be a
semigroup on a metric space }$(X,$\textit{\ }$d).$\textit{\ A compact set }$%
A\subset X$\textit{\ is called a global attractor for the semigroup }$%
\left\{ V(t)\right\} _{t\geq 0}$\textit{\ iff}
\end{definition}

$\bullet $\textit{\ }$A$\textit{\ is invariant, i.e. }$V(t)A=A,$\textit{\ }$%
\forall t\geq 0;$

$\bullet $\textit{\ }$\underset{t\rightarrow \infty }{\lim }$\textit{\ }$%
\underset{v\in B}{\sup }$\textit{\ }$\underset{u\in \mathcal{A}}{\inf }%
d(V(t)v,u)=0$\textit{\ \ for each bounded set }$B\subset X.$\textit{\newline
}

Our second result is the following theorem:

\begin{theorem}
Under conditions (2.2)-(2.4) the semigroup $\left\{ S(t)\right\} _{t\geq 0}$
generated by the problem (2.1) possesses a global attractor in $%
(H_{0}^{1}(\Omega )\cap L^{\infty }(\Omega ))\times L^{2}(\Omega )$.
\end{theorem}

\section{Well-posedness}

We start with the following existence result:

\begin{lemma}
Assume that the conditions (2.2)-(2.4) are satisfied. Then for every $%
(w_{0},w_{1})\in (H_{0}^{1}(\Omega )\cap L^{\infty }(\Omega ))\times
L^{2}(\Omega )$, the problem (2.1) admits a weak solution on $\left[ 0,T%
\right] \times \Omega $ such that
\begin{equation}
\underset{t\searrow 0}{\lim }\left\Vert w_{t}(t)-w_{1}\right\Vert
_{L^{2}(\Omega )}=0\text{.}  \tag{3.1}
\end{equation}
\end{lemma}

\begin{proof}
By using Galerkin's method, let us to construct approximate solutions of
(2.1). Let $\left\{ \varphi _{j}\right\} _{j=1}^{\infty }$ be a basis of $%
H^{2}(\Omega )\cap H_{0}^{1}(\Omega )$ consisting of the eigenfunctions of
the Dirichlet problem%
\begin{equation*}
\left\{
\begin{array}{c}
-\Delta \varphi _{j}=\lambda _{j}\varphi _{j}\text{ \ in }\Omega \\
\varphi _{j}=0\text{ \ on }\partial \Omega%
\end{array}%
\right. ,\text{ }j=1,2,...\text{ .}
\end{equation*}%
According to Lemma A.1 and the embedding $H^{2}(\Omega )\subset C(\overline{%
\Omega })$, for $(w_{0},w_{1})\in (H_{0}^{1}(\Omega )\cap L^{\infty }(\Omega
))\times L^{2}(\Omega )$, there exist the sequences $\left\{ \alpha
_{n}\right\} _{n=1}^{\infty }$ and $\left\{ \beta _{n}\right\}
_{n=1}^{\infty }$ such that%
\begin{equation}
\left\{
\begin{array}{c}
\sum\limits_{k=1}^{n}\alpha _{k}\varphi _{k}\rightarrow w_{0}\text{ in }%
H_{0}^{1}(\Omega )\text{, \ }\sum\limits_{k=1}^{n}\beta _{k}\varphi
_{k}\rightarrow w_{1}\text{ in }L^{2}(\Omega )\text{ as }n\rightarrow \infty
\text{,} \\
\underset{n}{\sup }\left\Vert \sum\limits_{k=1}^{n}\alpha _{k}\varphi
_{k}\right\Vert _{C(\overline{\Omega })}\leq \left\Vert w_{0}\right\Vert
_{L^{\infty }(\Omega )}\text{.}%
\end{array}%
\right.  \tag{3.2}
\end{equation}%
We define the approximate solution $w_{n}(t)$ in the form%
\begin{equation*}
w_{n}(t)=\sum\limits_{k=1}^{n}c_{nk}(t)\varphi _{k}\text{,}
\end{equation*}%
where $c_{nk}(t)$ are determined by the system of second order ordinary
differential equations
\begin{equation}
\begin{array}{c}
\left\langle \sum\limits_{k=1}^{n}c_{nk}^{\prime \prime }(t)\varphi
_{k},\varphi _{j}\right\rangle +\left\langle
\sum\limits_{k=1}^{n}c_{nk}^{\prime }(t)\nabla \varphi _{k},\nabla \varphi
_{j}\right\rangle +\left\langle f\left( \sum\limits_{k=1}^{n}c_{nk}^{\prime
}(t)\varphi _{k}\right) ,\varphi _{j}\right\rangle \\
+\left\langle \sum\limits_{k=1}^{n}c_{nk}(t)\nabla \varphi _{k},\nabla
\varphi _{j}\right\rangle +\left\langle g\left(
\sum\limits_{k=1}^{n}c_{nk}(t)\varphi _{k}\right) ,\varphi _{j}\right\rangle
=\left\langle h,\varphi _{j}\right\rangle \text{, \ }j=1,2,...,n,%
\end{array}
\tag{3.3}
\end{equation}%
with the initial data
\begin{equation}
c_{nj}(0)=\alpha _{j}\text{, \ \ }c_{nj}^{\prime }(0)=\beta _{j}\text{, \ \
\ }j=1,2,...,n\text{.}  \tag{3.4}
\end{equation}%
Since $det(\left\langle \varphi _{j},\varphi _{k}\right\rangle )\neq 0$ and
the nonlinear functions $f$ and $g$ are continuous, by the Peano existence
theorem, there exists at least one local solution to (3.3)-(3.4) in the
interval $[0,T_{n})$. Hence this allows to construct the approximate
solution $w_{n}(t)$. Multiplying the equation (3.3)$_{j}$ by the function $%
c_{nj}^{\prime }(t)$, summing from $j=1$ to $n$ and integrating over $(0,t),$
we have%
\begin{equation*}
\frac{1}{2}\left\Vert w_{nt}(t)\right\Vert _{L^{2}(\Omega )}^{2}+\frac{1}{2}%
\left\Vert \nabla w_{n}(t)\right\Vert _{L^{2}(\Omega )}^{2}+\left\langle
G(w_{n}(t)),1\right\rangle +\underset{0}{\overset{t}{\int }}\left\Vert
\nabla w_{nt}(\tau )\right\Vert _{L^{2}(\Omega )}^{2}d\tau
\end{equation*}%
\begin{equation*}
+\underset{0}{\overset{t}{\int }}\left\langle f(w_{nt}(\tau )),w_{nt}(\tau
)\right\rangle d\tau =\frac{1}{2}\left\Vert w_{nt}(0)\right\Vert
_{L^{2}(\Omega )}^{2}+\left\langle h,w_{n}(t)\right\rangle
\end{equation*}%
\begin{equation}
+\frac{1}{2}\left\Vert \nabla w_{n}(0)\right\Vert _{L^{2}(\Omega
)}^{2}+\left\langle G(w_{n}(0)),1\right\rangle -\left\langle
h,w_{n}(0)\right\rangle ,\text{ \ \ }\forall t\in \lbrack 0,T_{m}),
\tag{3.5}
\end{equation}%
where $G(w)=\underset{0}{\overset{w}{\int }}g(u)du$. Taking into account
(2.2), (2.3) and (3.2) in (3.5), we get%
\begin{equation*}
\left\{
\begin{array}{c}
\left\Vert w_{n}(t)\right\Vert _{H^{1}(\Omega )}^{2}\leq c, \\
\left\Vert w_{nt}(t)\right\Vert _{L^{2}(\Omega )}^{2}\leq c, \\
\underset{0}{\overset{t}{\int }}\left\Vert w_{nt}(s)\right\Vert
_{H^{1}(\Omega )}^{2}ds\leq c, \\
\underset{0}{\overset{t}{\int }}\left\langle f(w_{nt}(\tau ))+\lambda
_{1}w_{nt}(\tau ),w_{nt}(\tau )\right\rangle d\tau \leq c%
\end{array}%
\right. ,\text{ \ }\forall t\in \lbrack 0,T_{n})\text{.}
\end{equation*}%
where the constant $c$ depends on $\left\Vert (w_{0},w_{1})\right\Vert
_{(H_{0}^{1}(\Omega )\cap L^{\infty }(\Omega ))\times L^{2}(\Omega )}$ and
is independent of $n$ and $t$. Hence, we can extend the approximate solution
to the interval $[0,\infty )$ and
\begin{equation*}
\left\Vert w_{n}(t)\right\Vert _{H^{1}(\Omega )}^{2}+\left\Vert
w_{nt}(t)\right\Vert _{L^{2}(\Omega )}^{2}+\underset{0}{\overset{t}{\int }}%
\left\Vert w_{nt}(s)\right\Vert _{H^{1}(\Omega )}^{2}ds
\end{equation*}%
\begin{equation}
+\underset{0}{\overset{t}{\int }}\left\langle f_{1}(w_{nt}(\tau
)),w_{nt}(\tau )\right\rangle d\tau \leq 4c,\text{ \ }\forall t\geq 0\text{,}
\tag{3.6}
\end{equation}%
where $f_{1}(s)=f(s)+\lambda _{1}s$. Taking into account the embedding $%
H_{0}^{1}(\Omega )\subset L_{\Psi }^{\ast }(\Omega )$ (see \cite{21}), by
the last inequality, we have
\begin{equation*}
\underset{\Omega }{\int }\Psi (kw_{n}(t,x))dx\leq c_{1},\text{ }\forall
t\geq 0\text{,}
\end{equation*}%
for some $k>0$, where $L_{\Psi }^{\ast }(\Omega )$ is the Orlicz space with
the $N$-function $\Psi (z)=e^{z^{2}}$. The last inequality together with
(2.3) gives us
\begin{equation}
\left\Vert g(w_{n}(t))\right\Vert _{L^{2}(\Omega )}\leq c_{2}\text{, \ }%
\forall t\geq 0\text{,}  \tag{3.7}
\end{equation}%
where the constant $c_{2}$, as $c$ and $c_{1}$, depends on $\left\Vert
(w_{0},w_{1})\right\Vert _{(H_{0}^{1}(\Omega )\cap L^{\infty }(\Omega
))\times L^{2}(\Omega )}$ and is independent of $n$.

Now, multiplying the equation (3.3)$_{j}$, by the function $\frac{1}{\lambda
_{j}^{2}}c_{nj}^{\prime \prime }(t)$ and then summing from $j=1$ to $n$, we
get%
\begin{equation*}
\left\Vert (-\Delta )^{-1}w_{ntt}(t)\right\Vert _{L^{2}(\Omega )}^{2}
\end{equation*}%
\begin{equation*}
=\left\langle w_{nt}(t)+w_{n}(t)-(-\Delta )^{-1}f(w_{nt}(t))-(-\Delta
)^{-1}g(w_{n}(t))+(-\Delta )^{-1}h,(-\Delta )^{-1}w_{ntt}(t)\right\rangle
\text{,}
\end{equation*}%
and consequently%
\begin{equation*}
\left\Vert w_{ntt}(t)\right\Vert _{V}\leq c_{3}\left( \left\Vert
w_{nt}(t)\right\Vert _{L^{2}(\Omega )}+\left\Vert w_{n}(t)\right\Vert
_{L^{2}(\Omega )}\right.
\end{equation*}%
\begin{equation}
+\left. \left\Vert h\right\Vert _{L^{2}(\Omega )}+\left\Vert
f(w_{nt}(t))\right\Vert _{V}+\left\Vert g(w_{n}(t))\right\Vert
_{L^{2}(\Omega )}\right) \text{, \ }\forall t\geq 0\text{,}  \tag{3.8}
\end{equation}%
where $V$ is the dual of $H^{2}(\Omega )\cap H_{0}^{1}(\Omega )$. Since, by
the embedding $L^{1}(\Omega )\subset V$,%
\begin{equation*}
\left\Vert f(w_{nt}(t))\right\Vert _{V}\leq c_{4}\left\Vert
f(w_{nt}(t))\right\Vert _{L^{1}(\Omega )}\leq c_{4}\left\Vert f\right\Vert
_{C[-1,1]}mes(\Omega )
\end{equation*}%
\begin{equation*}
+c_{4}\underset{\left\{ x:x\in \Omega ,\text{ }\left\vert wt,x)\right\vert
>1\right\} }{\int }\left\vert f(w_{nt}(t,x))\right\vert \left\vert
w_{nt}(t,x)\right\vert dx\leq c_{4}\left\Vert f\right\Vert
_{C[-1,1]}mes(\Omega )
\end{equation*}%
\begin{equation*}
+c_{4}\left\langle f_{1}(w_{nt}(t)),w_{nt}(t)\right\rangle +c_{5}\left\Vert
w_{nt}(t)\right\Vert _{L^{2}(\Omega )}^{2}
\end{equation*}%
\begin{equation*}
\leq c_{6}(1+\left\Vert w_{nt}(t)\right\Vert _{L^{2}(\Omega
)}^{2}+\left\langle f_{1}(w_{nt}(t)),w_{nt}(t)\right\rangle )
\end{equation*}%
by (3.6)-(3.8), we have%
\begin{equation}
\underset{0}{\overset{1}{\int }}\left\Vert w_{ntt}(s)\right\Vert _{V}ds\leq
c_{8}\text{,}  \tag{3.9}
\end{equation}%
where $c_{8}$ is independent of $n$. Thus it follows from the estimates
(3.6) and (3.9) that the sequences $\left\{ w_{n}\right\} _{n=1}^{\infty }$,
$\left\{ w_{nt}\right\} _{n=1}^{\infty }$ and $\left\{ w_{ntt}\right\}
_{n=1}^{\infty }$ are bounded in $L^{\infty }(0,T;H_{0}^{1}(\Omega ))$, $%
L^{\infty }(0,T;L^{2}(\Omega ))\cap L^{2}(0,T;H_{0}^{1}(\Omega ))$ and $%
L^{1}(0,T;V)$, respectively. Applying the compact embedding theorem (see
\cite[Corollary 4]{22}), we obtain that the sequences $\left\{ w_{n}\right\}
_{n=1}^{\infty }$and $\left\{ w_{nt}\right\} _{n=1}^{\infty }$ are
precompact in $L^{2}(0,T;L^{2}(\Omega ))$. As a consequence, there exists a
subsequence of $\left\{ w_{n}\right\} _{n=1}^{\infty }$ (still denoted by $%
\left\{ w_{n}\right\} _{n=1}^{\infty }$) and a function $w\in L^{\infty
}(0,T;H_{0}^{1}(\Omega ))$ with $w_{t}\in $ $L^{\infty }(0,T;L^{2}(\Omega
))\cap L^{2}(0,T;H_{0}^{1}(\Omega ))$ such that%
\begin{equation}
\left\{
\begin{array}{c}
w_{n}\rightarrow w\text{ \ weakly star in }L^{\infty }(0,T;H_{0}^{1}(\Omega
))\text{,} \\
w_{n}\rightarrow w\text{ \ strongly in }C(\left[ 0,T\right] ;L^{2}(\Omega ))%
\text{,} \\
w_{nt}\rightarrow w_{t}\text{ \ weakly star in }L^{\infty }(0,T;L^{2}(\Omega
))\text{,} \\
w_{nt}\rightarrow w_{t}\text{ \ weakly in }L^{2}(0,T;H_{0}^{1}(\Omega ))%
\text{,} \\
w_{nt}\rightarrow w_{t}\text{ \ strongly in }L^{2}(0,T;L^{2}(\Omega ))\text{,%
} \\
w_{n}\rightarrow w\text{ \ a.e. in }(0,T)\times \Omega \text{,} \\
w_{nt}\rightarrow w_{t}\text{ \ a.e. in }(0,T)\times \Omega \text{.}%
\end{array}%
\right.  \tag{3.10}
\end{equation}%
Now, applying Lemma A.2, by (3.6) and (3.10)$_{7}$, we get%
\begin{equation*}
f_{1}(w_{nt})\rightarrow f_{1}(w_{t})\text{ strongly in }L^{1}(0,T;L^{1}(%
\Omega ))\text{,}
\end{equation*}%
which, together with (3.10)$_{5}$, implies%
\begin{equation}
f(w_{nt})\rightarrow f(w_{t})\text{ strongly in }L^{1}(0,T;L^{1}(\Omega ))%
\text{.}  \tag{3.11}
\end{equation}%
Also, by (3.7) and (3.10)$_{6}$, we find%
\begin{equation}
g(w_{n})\rightarrow g(w)\text{ weakly in }L^{2}(0,T;L^{2}(\Omega ))\text{.}
\tag{3.12}
\end{equation}%
Thus, considering the last approximations and passing to the limit in (3.3),
we obtain%
\begin{equation}
\frac{d}{dt}\left\langle w_{t},\varphi _{j}\right\rangle +\left\langle
\nabla w_{t},\nabla \varphi _{j}\right\rangle +\left\langle \nabla w,\nabla
\varphi _{j}\right\rangle +\left\langle f(w_{t}),\varphi _{j}\right\rangle
+\left\langle g(w),\varphi _{j}\right\rangle =\left\langle h,\varphi
_{j}\right\rangle ,\text{ \ }j=1,2,...\text{.}  \tag{3.13}
\end{equation}%
in the sense of distributions on $(0,T)$. Since $\left\{ \varphi
_{j}\right\} _{j=1}^{\infty }$ is base in $H^{2}(\Omega )\cap
H_{0}^{1}(\Omega )$, by (3.2), for every $\upsilon \in H_{0}^{1}(\Omega
)\cap L^{\infty }(\Omega )$, there exists $\left\{ \alpha _{nj}\right\}
_{j=1}^{k_{n}}$ such that%
\begin{equation*}
\sum\limits_{j=1}^{k_{n}}\alpha _{nj}\varphi _{j}\rightarrow v\text{ in }%
H_{0}^{1}(\Omega )\text{, as }n\rightarrow \infty \text{, \ \ and }\underset{%
n}{\sup }\left\Vert \sum\limits_{j=1}^{k_{n}}\alpha _{nj}\varphi
_{j}\right\Vert _{C(\overline{\Omega })}\leq \left\Vert v\right\Vert
_{L^{\infty }(\Omega )}\text{,}
\end{equation*}%
which, together with (3.13), yields that%
\begin{equation*}
\frac{d}{dt}\left\langle w_{t},v\right\rangle +\left\langle \nabla
w_{t},\nabla v\right\rangle +\left\langle \nabla w,\nabla v\right\rangle
+\left\langle f(w_{t}),v\right\rangle +\left\langle g(w),v\right\rangle
=\left\langle h,v\right\rangle \text{,}
\end{equation*}%
in the sense of distributions on $(0,T)$, for every $\upsilon \in
H_{0}^{1}(\Omega )\cap L^{\infty }(\Omega )$. From this equation, it is easy
to see that $w_{tt}\in L^{1}(0,T;H^{-1}(\Omega )+L^{1}(\Omega ))$ and
\begin{equation*}
w_{tt}-\Delta w_{t}-\Delta w+f(w_{t})+g(w)=h\text{ \ in }L^{1}(0,T;H^{-1}(%
\Omega )+L^{1}(\Omega ))\text{.}
\end{equation*}%
Now, by $w\in L^{\infty }(0,T;H_{0}^{1}(\Omega ))$, $w_{t}\in
L^{2}(0,T;H_{0}^{1}(\Omega ))$ and $w_{tt}\in L^{1}(0,T;H^{-1}(\Omega
)+L^{1}(\Omega ))$, it follows that $w\in C(\left[ 0,T\right]
;H_{0}^{1}(\Omega ))$ and $w_{t}\in C(\left[ 0,T\right] ;H^{-1}(\Omega
)+L^{1}(\Omega ))$. Further, applying \cite[Lemma 8.1, p. 275]{23}, by $%
w_{t}\in L^{\infty }(0,T;L^{2}(\Omega ))$ and $w_{t}\in C(\left[ 0,T\right]
;H^{-1}(\Omega )+L^{1}(\Omega ))$, we have $w_{t}\in C_{s}(0,T;L^{2}(\Omega
))$.

By (3.10)$_{1}$ and (3.10)$_{4}$, it follows that
\begin{equation*}
w_{n}\rightarrow w\text{ \ weakly in }C(\left[ 0,T\right] ;H_{0}^{1}(\Omega
))
\end{equation*}%
and particularly
\begin{equation*}
w_{n}(0)\rightarrow w(0)\text{ \ weakly in }H_{0}^{1}(\Omega )\text{,}
\end{equation*}%
which, together with (3.2)$_{1}$, yields that $w(0)=w_{0}$. Also, by (3.3),
(3.10)-(3.12) and (3.13), we find
\begin{equation*}
\frac{d}{dt}\left\langle w_{nt},\varphi _{j}\right\rangle \rightarrow \frac{d%
}{dt}\left\langle w_{t},\varphi _{j}\right\rangle \text{ weakly in }%
L^{1}(0,T),\text{ }j=1,2,...\text{.}
\end{equation*}%
The last approximation, using (3.10)$_{3}$, gives us%
\begin{equation*}
\left\langle w_{nt}(0),\varphi _{j}\right\rangle \rightarrow \left\langle
w_{t}(0),\varphi _{j}\right\rangle ,\text{ \ }j=1,2,...\text{.}
\end{equation*}%
Hence, by (3.2)$_{1}$ and (3.4), we have $w_{t}(0)=w_{1}$. Thus, the
function $w\in C(\left[ 0,T\right] ;H_{0}^{1}(\Omega ))$ with $w_{t}\in
C_{s}(0,T;L^{2}(\Omega ))$ is a weak solution to (2.1) on $\left[ 0,T\right]
\times \Omega $.

Now, let us prove (3.1). By (3.2), (3.4) and Lebesgue's convergence theorem,
it is easy to see that
\begin{equation*}
\underset{n\rightarrow \infty }{\lim }\left\langle
F(w_{n}(0)),1\right\rangle =\left\langle F(w_{0}),1\right\rangle \text{.}
\end{equation*}%
Also, by (2.2), (2.3), (3.10) and Fatou's lemma, we have%
\begin{equation*}
\underset{n\rightarrow \infty }{\lim \inf }\left\langle
G(w_{n}(t)),1\right\rangle \geq \left\langle G(w(t)),1\right\rangle
\end{equation*}%
and
\begin{equation*}
\underset{n\rightarrow \infty }{\lim \inf }\underset{0}{\overset{t}{\int }}%
\left\langle f(w_{nt}(\tau )),w_{nt}(\tau )\right\rangle d\tau \geq \underset%
{0}{\overset{t}{\int }}\left\langle f(w_{t}(\tau )),w_{t}(\tau
)\right\rangle d\tau \text{.}
\end{equation*}%
Hence, passing to the limit in (3.5) and taking the weak lower
semi-continuity of the norm into consideration leads to
\begin{equation*}
\frac{1}{2}\left\Vert w_{t}(t)\right\Vert _{L^{2}(\Omega )}^{2}+\frac{1}{2}%
\left\Vert \nabla w(t)\right\Vert _{L^{2}(\Omega )}^{2}+\left\langle
G(w(t)),1\right\rangle +\underset{0}{\overset{t}{\int }}\left\Vert \nabla
w_{t}(\tau )\right\Vert _{L^{2}(\Omega )}^{2}d\tau
\end{equation*}%
\begin{equation*}
+\underset{0}{\overset{t}{\int }}\left\langle f(w_{t}(\tau )),w_{t}(\tau
)\right\rangle d\tau \leq \frac{1}{2}\left\Vert w_{1}\right\Vert
_{L^{2}(\Omega )}^{2}+\left\langle h,w(t)\right\rangle
\end{equation*}%
\begin{equation}
+\frac{1}{2}\left\Vert \nabla w_{0}\right\Vert _{L^{2}(\Omega
)}^{2}+\left\langle G(w_{0}),1\right\rangle -\left\langle
h,w_{0}\right\rangle ,\text{ \ \ }\forall t\in \lbrack 0,T]\text{.}
\tag{3.14}
\end{equation}%
Since $w\in C(\left[ 0,T\right] ;H_{0}^{1}(\Omega ))$, passing to the limit
in (3.14) as $t\searrow 0$, we get%
\begin{equation*}
\underset{n\rightarrow \infty }{\lim \inf }\left\Vert w_{t}(t)\right\Vert
_{L^{2}(\Omega )}^{2}\leq \left\Vert w_{1}\right\Vert _{L^{2}(\Omega )}^{2}%
\text{,}
\end{equation*}%
which, together with $w_{t}\in C_{s}(0,T;L^{2}(\Omega ))$, yields (3.1).
\end{proof}

Now, let us prove $L^{\infty }$ regularity for the weak solutions. Decompose
the weak solution determined by Lemma 3.1 as follows%
\begin{equation*}
w(t,x)=v(t,x)+u(t,x)\text{,}
\end{equation*}%
where
\begin{equation}
\left\{
\begin{array}{c}
v_{tt}-\Delta v_{t}+v_{t}-\Delta v=(1+\lambda _{1})w_{t}-g(w)+h\text{\ \ \ \
in \ }(0,T)\times \Omega \text{,} \\
v=0\text{ \ \ \ \ \ \ \ \ \ \ \ \ \ \ \ \ \ \ \ \ \ \ \ \ \ \ \ \ \ \ \ \ \
\ \ \ \ \ \ \ \ \ \ \ \ \ \ \ \ \ \ on \ }(0,T)\times \partial \Omega \text{,%
} \\
v(0,\cdot )=w_{0},\text{ \ \ \ \ }v_{t}(0,\cdot )=w_{1}\text{\ \ \ \ \ \ \ \
\ \ \ \ \ \ \ \ \ \ \ \ \ \ \ \ \ \ \ \ \ \ \ \ \ \ \ \ in \ }\Omega \text{ }%
\end{array}%
\right.  \tag{3.15}
\end{equation}%
and

\begin{equation}
\left\{
\begin{array}{c}
u_{tt}-\Delta u_{t}+u_{t}-\Delta u=-f_{1}(w_{t})\text{\ \ \ \ in \ }%
(0,T)\times \Omega \text{, \ \ \ \ \ \ \ \ \ \ } \\
u=0\text{ \ \ \ \ \ \ \ \ \ \ \ \ \ \ \ \ \ \ \ \ \ \ \ \ \ \ \ \ \ \ \ \ \
on \ }(0,T)\times \partial \Omega \text{, \ \ \ \ \ \ \ \ \ \ } \\
u(0,\cdot )=0,\text{ \ \ \ \ \ \ \ \ \ \ }u_{t}(0,\cdot )=0\text{\ \ \ \ \ \
\ \ \ \ \ \ \ \ \ in \ }\Omega \text{. \ \ \ \ \ \ \ \ \ \ \ \ }%
\end{array}%
\right.  \tag{3.16}
\end{equation}

\begin{lemma}
Let $(w_{0},w_{1})\in (H_{0}^{1}(\Omega )\cap L^{\infty }(\Omega ))\times
L^{2}(\Omega )$ and $w(t,x)$ be the weak solution of the problem (2.1). Then
the problem (3.15) has a unique weak solution $v\in C(\left[ 0,T\right]
;H_{0}^{1}(\Omega )\cap L^{\infty }(\Omega ))$ with $v_{t}\in C(\left[ 0,T%
\right] ;L^{2}(\Omega ))\cap L^{2}(0,T;H_{0}^{1}(\Omega ))$ such that%
\begin{equation*}
\left\Vert v(t)-e^{-t}w_{0}\right\Vert _{H^{s}(\Omega )}\leq \frac{c}{2-s}%
\left( 1+\left\Vert g(w)\right\Vert _{L^{\infty }(0,T;L^{2}(\Omega ))}\right.
\end{equation*}%
\begin{equation}
+\left. \left\Vert w_{t}\right\Vert _{L^{\infty }(0,T;L^{2}(\Omega
))}\right) ,\text{ \ }\forall t\in \lbrack 0,T],\text{ \ }\forall s\in
\lbrack 0,2)\text{.}  \tag{3.17}
\end{equation}
\end{lemma}

\begin{proof}
Denoting $\varphi =v+v_{t}$, by (3.15), we have%
\begin{equation}
\left\{
\begin{array}{c}
\varphi _{t}-\Delta \varphi =(1+\lambda _{1})w_{t}-g(w)+h\text{\ \ \ \ in \ }%
(0,T)\times \Omega \text{,} \\
\varphi =0\text{ \ \ \ \ \ \ \ \ \ \ \ \ \ \ \ \ \ \ \ \ \ \ \ \ \ \ \ \ \ \
\ \ \ \ \ on \ }(0,T)\times \partial \Omega \text{,} \\
\varphi (0,\cdot )=w_{0}+w_{1}\text{ \ \ \ \ \ \ \ \ \ \ \ \ \ \ \ \ \ \ \ \
\ \ \ \ \ \ \ \ \ \ \ \ \ \ in \ }\Omega \text{.}%
\end{array}%
\right.  \tag{3.18}
\end{equation}%
It is well known (see, for example \cite[p. 116]{24}) that the Laplace
operator $\Delta $ with $D(\Delta )=H^{2}(\Omega )\cap H_{0}^{1}(\Omega )$
generates an analitic semigroup in $L^{2}(\Omega )$ and
\begin{equation}
\left\Vert e^{\Delta t}\right\Vert _{\mathcal{L(}L^{2}(\Omega ),H^{s}(\Omega
))}\leq Mt^{-\frac{s}{2}}\text{,}  \tag{3.19}
\end{equation}%
for $t,s\geq 0$. Hence, by using the variations of constants formula, from
(3.18) and (3.19), we obtain that $\varphi \in C(\left[ 0,T\right]
;L^{2}(\Omega ))\cap C((0,T];H^{s}(\Omega ))\cap L^{2}(0,T;H_{0}^{1}(\Omega
))$ and
\begin{equation}
\left\Vert \varphi (t)\right\Vert _{H^{s}(\Omega )}\leq Mt^{-\frac{s}{2}%
}\left\Vert w_{0}+w_{1}\right\Vert _{L^{2}(\Omega )}+\frac{2}{2-s}%
M(1+\left\Vert g(w)\right\Vert _{L^{\infty }(0,T;L^{2}(\Omega ))}+\left\Vert
w_{t}\right\Vert _{L^{\infty }(0,T;L^{2}(\Omega ))})\text{,}  \tag{3.20}
\end{equation}%
for every $t\in \lbrack 0,T]$ and $s\in \lbrack 0,2)$. As a consequence,
solving the equation $v+v_{t}=\varphi $, we get (3.17). By using the
embedding $H^{1+\varepsilon }(\Omega )\subset L^{\infty }(\Omega )$, from
(3.17), (3.20) and the equation $v+v_{t}=\varphi $, it follows that $v\in
L^{\infty }(0,T;H_{0}^{1}(\Omega )\cap L^{\infty }(\Omega ))$ and $v_{t}\in
L^{1}(0,T;H_{0}^{1}(\Omega )\cap L^{\infty }(\Omega ))$. Hence, $v\in C(%
\left[ 0,T\right] ;H_{0}^{1}(\Omega )\cap L^{\infty }(\Omega ))$, which,
together with $\varphi \in C(\left[ 0,T\right] ;L^{2}(\Omega ))\cap
L^{2}(0,T;H_{0}^{1}(\Omega ))$, yields $v_{t}\in C(\left[ 0,T\right]
;L^{2}(\Omega ))\cap L^{2}(0,T;H_{0}^{1}(\Omega ))$.
\end{proof}

\begin{lemma}
Assume that the conditions (2.2)-(2.4) are satisfied and $w(t,x)$ is the
weak solution of the problem (2.1) on $\left[ 0,T\right] \times \Omega $.
Then $u\in C(\left[ 0,T\right] ;L^{\infty }(\Omega ))$ and for every $\delta
>0$ there exists $c(\delta ,T)>0$ such that%
\begin{equation}
\left\Vert u(t)\right\Vert _{L^{\infty }(\Omega )}\leq \delta +c(\delta ,T)%
\underset{0}{\overset{T}{\int }}\left\langle
f_{1}(w_{t}(s)),w_{t}(s)\right\rangle ds,\text{ \ }\forall t\in \lbrack 0,T]%
\text{, }  \tag{3.21}
\end{equation}%
where $u\in C(\left[ 0,T\right] ;H_{0}^{1}(\Omega ))$, with $u_{t}\in
C_{s}(0,T;L^{2}(\Omega ))\cap L^{2}(0,T;H_{0}^{1}(\Omega ))$ and $\underset{%
t\searrow 0}{\lim }\left\Vert u_{t}(t)\right\Vert _{L^{2}(\Omega )}=0$, is
the weak solution of (3.16).
\end{lemma}

\begin{proof}
Setting $h(t,x)=-f_{1}(w_{t}(t,x))$ and $v=u+u_{t}$, by (3.16), we have%
\begin{equation*}
\left\{
\begin{array}{c}
v_{t}-\Delta v=h(t,x)\text{\ \ \ \ in \ }(0,T)\times \Omega \text{,} \\
v=0\text{ \ \ \ \ \ \ \ \ \ \ \ \ \ \ \ on \ }(0,T)\times \partial \Omega
\text{,} \\
v(0,\cdot )=0\text{ \ \ \ \ \ \ \ \ \ \ \ \ \ \ \ \ \ \ \ \ \ \ in \ }\Omega
\text{.}%
\end{array}%
\right.
\end{equation*}%
By Lemma A.3, it follows that%
\begin{equation*}
\left\vert v(\tau ,x)\right\vert \leq \text{ }\frac{1}{4\pi }\underset{0}{%
\overset{\tau }{\int }}\frac{1}{(\tau -s)}\underset{\Omega }{\int }e^{-\frac{%
\left\vert x-y\right\vert ^{2}}{4(\tau -s)}}\left\vert h(s,y)\right\vert dyds
\end{equation*}%
\begin{equation}
\text{ }=\frac{1}{4\pi }\underset{0}{\overset{\tau }{\int }}\frac{1}{(\tau
-s)}\underset{\Omega }{\int }e^{-\frac{\left\vert x-y\right\vert ^{2}}{%
4(\tau -s)}}\left\vert f_{1}(w_{t}(s,y))\right\vert dyds\text{, \ a.e. in }%
(0,T)\times \Omega \text{.}  \tag{3.22}
\end{equation}%
Now, let us estimate the right hand side of (3.22).%
\begin{equation*}
\underset{0}{\overset{\tau }{\int }}\frac{1}{(\tau -s)}\underset{\Omega }{%
\int }e^{-\frac{\left\vert x-y\right\vert ^{2}}{4(\tau -s)}}\left\vert
f_{1}(w_{t}(s,y))\right\vert dyds
\end{equation*}%
\begin{equation*}
=\underset{0}{\overset{\tau }{\int }}\frac{1}{(\tau -s)}\underset{\left\{
x\in \Omega :f_{1}^{-1}(-\varepsilon (\tau -s)^{-\alpha })\leq
w_{t}(s,y)\leq f_{1}^{-1}(\varepsilon (\tau -s)^{-\alpha })\right\} }{\int }%
e^{-\frac{\left\vert x-y\right\vert ^{2}}{4(\tau -s)}}\left\vert
f_{1}(w_{t}(s,y))\right\vert dyds
\end{equation*}%
\begin{equation*}
+\underset{0}{\overset{\tau }{\int }}\frac{1}{(\tau -s)}\underset{\left\{
x\in \Omega :w_{t}(s,y)>f_{1}^{-1}(\varepsilon (\tau -s)^{-\alpha })\right\}
}{\int }e^{-\frac{\left\vert x-y\right\vert ^{2}}{4(\tau -s)}}\left\vert
f_{1}(w_{t}(s,y))\right\vert dyds
\end{equation*}%
\begin{equation*}
+\underset{0}{\overset{\tau }{\int }}\frac{1}{\tau -s}\underset{\left\{ x\in
\Omega :w_{t}(s,y)<f_{1}^{-1}(-\varepsilon (\tau -s)^{-\alpha })\right\} }{%
\int }e^{-\frac{\left\vert x-y\right\vert ^{2}}{4(\tau -s)}}\left\vert
f_{1}(w(s,y))\right\vert dyds
\end{equation*}%
\begin{equation*}
\leq \varepsilon \underset{0}{\overset{\tau }{\int }}\frac{1}{(\tau
-s)^{1+\alpha }}\underset{R^{2}}{\int }e^{-\frac{\left\vert x-y\right\vert
^{2}}{4(\tau -s)}}dyds+\underset{0}{\overset{\tau }{\int }}\frac{1}{\left(
\tau -s\right) f_{1}^{-1}(\varepsilon (\tau -s)^{-\alpha })}\underset{\Omega
}{\int }f_{1}(w_{t}(s,y))w_{t}(s,y)dyds
\end{equation*}%
\begin{equation*}
-\underset{0}{\overset{\tau }{\int }}\frac{1}{\left( \tau -s\right)
f_{1}^{-1}(-\varepsilon (\tau -s)^{-\alpha })}\underset{\Omega }{\int }%
f_{1}(w_{t}(s,y))w_{t}(s,y)dyds
\end{equation*}%
\begin{equation*}
=4\pi \varepsilon \underset{0}{\overset{\tau }{\int }}\frac{1}{(\tau
-s)^{\alpha }}ds+\underset{0}{\overset{\tau }{\int }}\frac{1}{\left( \tau
-s\right) f_{1}^{-1}(\varepsilon (\tau -s)^{-\alpha })}\underset{\Omega }{%
\int }f_{1}(w_{t}(s,y))w_{t}(s,y)dyds
\end{equation*}%
\begin{equation*}
-\underset{0}{\overset{\tau }{\int }}\frac{1}{\left( \tau -s\right)
f_{1}^{-1}(-\varepsilon (\tau -s)^{-\alpha })}\underset{\Omega }{\int }%
f_{1}(w_{t}(s,y))w_{t}(s,y)dyds,\text{ \ }\forall (\tau ,x)\in (0,T)\times
\Omega ,
\end{equation*}%
where $\alpha \in (0,1)$. Hence, we have%
\begin{equation*}
\underset{0}{\overset{t}{\int }}\left\Vert \underset{0}{\overset{\tau }{\int
}}\frac{1}{(\tau -s)}\underset{\Omega }{\int }e^{-\frac{\left\vert
x-y\right\vert ^{2}}{4(\tau -s)}}\left\vert f_{1}(w_{t}(s,y))\right\vert
dyds\right\Vert _{L^{\infty }(\Omega )}d\tau
\end{equation*}%
\begin{equation*}
\leq \frac{4\pi \varepsilon }{(2-\alpha )(1-\alpha )}t^{2-\alpha }
\end{equation*}%
\begin{equation*}
+\underset{0}{\overset{t}{\int }}\underset{0}{\overset{\tau }{\int }}\left(
\frac{1}{\left( \tau -s\right) f_{1}^{-1}(\varepsilon (\tau -s)^{-\alpha })}-%
\frac{1}{\left( \tau -s\right) f_{1}^{-1}(-\varepsilon (\tau -s)^{-\alpha })}%
\right) \underset{\Omega }{\int }f_{1}(w_{t}(s,y))w_{t}(s,y)dydsd\tau
\end{equation*}%
\begin{equation*}
=\underset{0}{\overset{t}{\int }}\underset{\Omega }{\int }%
f_{1}(w_{t}(s,y))w_{t}(s,y)dyds\underset{s}{\overset{t}{\int }}\left( \frac{1%
}{\left( \tau -s\right) f_{1}^{-1}(\varepsilon (\tau -s)^{-\alpha })}-\frac{1%
}{\left( \tau -s\right) f_{1}^{-1}(-\varepsilon (\tau -s)^{-\alpha })}%
\right) d\tau ds
\end{equation*}%
\begin{equation}
+\frac{4\pi \varepsilon }{(2-\alpha )(1-\alpha )}t^{2-\alpha }\text{.}
\tag{3.23}
\end{equation}%
By the condition (2.4), we obtain%
\begin{equation*}
\underset{s}{\overset{t}{\int }}\left( \frac{1}{\left( \tau -s\right)
f_{1}^{-1}(\varepsilon (\tau -s)^{-\alpha })}-\frac{1}{\left( \tau -s\right)
f_{1}^{-1}(-\varepsilon (\tau -s)^{-\alpha })}\right) d\tau
\end{equation*}%
\begin{equation*}
=\underset{0}{\overset{t-s}{\int }}\left( \frac{1}{\sigma
f_{1}^{-1}(\varepsilon \sigma ^{-\alpha })}-\frac{1}{\sigma
f_{1}^{-1}(-\varepsilon \sigma ^{-\alpha })}\right) d\sigma \leq \frac{1}{%
\alpha }\underset{T^{-\alpha }}{\overset{\infty }{\int }}\frac{1}{\lambda
f_{1}^{-1}(\varepsilon \lambda )}d\lambda
\end{equation*}%
\begin{equation*}
-\frac{1}{\alpha }\underset{T^{-\alpha }}{\overset{\infty }{\int }}\frac{1}{%
\lambda f_{1}^{-1}(-\varepsilon \lambda )}d\lambda =\frac{1}{\alpha }%
\underset{\varepsilon T^{-\alpha }}{\overset{\infty }{\int }}\frac{1}{%
\lambda f_{1}^{-1}(\lambda )}d\lambda
\end{equation*}%
\begin{equation*}
-\frac{1}{\alpha }\underset{\varepsilon T^{-\alpha }}{\overset{\infty }{\int
}}\frac{1}{\lambda f_{1}^{-1}(-\lambda )}d\lambda =\frac{1}{\alpha }\underset%
{f_{1}^{-1}(\varepsilon T^{-\alpha })}{\overset{\infty }{\int }}\frac{%
f_{1}^{\prime }(\nu )}{\nu f_{1}(\nu )}d\nu
\end{equation*}%
\begin{equation*}
+\frac{1}{\alpha }\underset{-\infty }{\overset{f_{1}^{-1}(-\varepsilon
T^{-\alpha })}{\int }}\frac{f_{1}^{\prime }(\nu )}{\nu f_{1}(\nu )}d\nu
<\infty ,
\end{equation*}%
which, together with (3.23), yields%
\begin{equation*}
\underset{0}{\overset{t}{\int }}\left\Vert \underset{0}{\overset{\tau }{\int
}}\frac{1}{(\tau -s)}\underset{\Omega }{\int }e^{-\frac{\left\vert
x-y\right\vert ^{2}}{4(\tau -s)}}\left\vert f_{1}(w_{t}(s,y))\right\vert
dyds\right\Vert _{L^{\infty }(\Omega )}d\tau \leq \frac{4\pi \varepsilon }{%
(2-\alpha )(1-\alpha )}T^{2-\alpha }
\end{equation*}%
\begin{equation*}
+k_{\varepsilon ,T}\underset{0}{\overset{T}{\int }}\underset{\Omega }{\int }%
f_{1}(w_{t}(s,y))w_{t}(s,y)dyds,\text{ \ \ \ }\forall t\in \lbrack 0,T]\text{%
,}
\end{equation*}%
where $k_{\varepsilon ,T}=\frac{1}{\alpha }\left( \underset{%
f_{1}^{-1}(\varepsilon T^{-\alpha })}{\overset{\infty }{\int }}\frac{%
f_{1}^{\prime }(\nu )}{\nu f_{1}(\nu )}d\nu +\underset{-\infty }{\overset{%
f_{1}^{-1}(-\varepsilon T^{-\alpha })}{\int }}\frac{f_{1}^{\prime }(\nu )}{%
\nu f_{1}(\nu )}d\nu \right) $. Taking into account the last inequality in
(3.22), we get $v\in L^{1}(0,T;L^{\infty }(\Omega ))$ and
\begin{equation*}
\underset{0}{\overset{T}{\int }}\left\Vert v(\tau )\right\Vert _{L^{\infty
}(\Omega )}d\tau \leq \frac{2\varepsilon }{(2-\alpha )(1-\alpha )}%
T^{3-\alpha }
\end{equation*}%
\begin{equation}
+\frac{Tk_{\varepsilon ,T}}{4\pi }\underset{0}{\overset{T}{\int }}\underset{%
\Omega }{\int }f_{1}(w_{t}(s,y))w_{t}(s,y)dyds\text{, \ \ }\forall
\varepsilon >0\text{.}  \tag{3.24}
\end{equation}

Now, solving the problem%
\begin{equation*}
\left\{
\begin{array}{c}
u_{t}(t,x)+u(t,x)=v(t,x)\text{ \ in }(0,T)\times \Omega , \\
u(0,x)=0\text{ \ in }\Omega \text{\ \ }%
\end{array}%
\right.
\end{equation*}%
and taking into account (3.24), we obtain $u\in C(\left[ 0,T\right]
;L^{\infty }(\Omega )),$ $u_{t}\in L^{1}(0,T;L^{\infty }(\Omega ))$ and
\begin{equation*}
\left\Vert u(t)\right\Vert _{L^{\infty }(\Omega )}\leq \underset{0}{\overset{%
t}{\int }}e^{-(t-\tau )}\left\Vert v(\tau )\right\Vert _{L^{\infty }(\Omega
)}d\tau \leq \underset{0}{\overset{t}{\int }}\left\Vert v(\tau )\right\Vert
_{L^{\infty }(\Omega )}d\tau
\end{equation*}%
\begin{equation}
\leq \frac{2\varepsilon }{(2-\alpha )(1-\alpha )}T^{3-\alpha }+\frac{%
Tk_{\varepsilon ,T}}{4\pi }\underset{0}{\overset{T}{\int }}\left\langle
f_{1}(w_{t}(s)),w_{t}(s)\right\rangle ds\text{, \ \ }\forall t\in \lbrack
0,T]\text{ and }\forall \varepsilon >0\text{,}  \tag{3.25}
\end{equation}%
which yields (3.21).
\end{proof}

Thus, by Lemma 3.1-3.3, it follows that the problem (2.1) has a weak
solution $w\in C(\left[ 0,T\right] ;H_{0}^{1}(\Omega )\cap L^{\infty
}(\Omega ))$ with $w_{t}\in C_{s}(0,T;L^{2}(\Omega ))\cap
L^{2}(0,T;H_{0}^{1}(\Omega ))$. Also, by (3.7), (3.10), (3.14), (3.17) and
(3.21), we have%
\begin{equation}
\left\Vert w(t)\right\Vert _{L^{\infty }(\Omega )}\leq c(T,\left\Vert
(w_{0},w_{1})\right\Vert _{(H_{0}^{1}(\Omega )\cap L^{\infty }(\Omega
))\times L^{2}(\Omega )}),\text{ }\forall t\in \lbrack 0,T]\text{,}
\tag{3.26}
\end{equation}%
where $c:R_{+}\times R_{+}\rightarrow R_{+}$ is a nondecreasing function
with respect to each variable.

Now, we are in a position to prove the uniqueness of the weak solution.

\begin{lemma}
Under the conditions (2.2)-(2.4), the problem (2.1) has a unique weak
solution.
\end{lemma}

\begin{proof}
Assume that there are two solutions to (2.1) as $w^{(i)}\in C(\left[ 0,T%
\right] ;H_{0}^{1}(\Omega )\cap L^{\infty }(\Omega ))$ with $w_{t}^{(i)}\in
C_{s}(0,T;L^{2}(\Omega ))\cap L^{2}(0,1;H_{0}^{1}(\Omega ))$, $%
w^{(i)}(0)=w_{0}$ and $\underset{t\searrow 0}{\lim }\left\Vert
w_{t}^{(i)}(t)-w_{1}\right\Vert _{L^{2}(\Omega )}=0$, $i=1,2$. Let $%
u(t,x)=w_{1}(t,x)-w_{2}(t,x)$. Testing the equation
\begin{equation*}
u_{tt}-\Delta u_{t}+f(w_{t}^{(1)})-f(w_{t}^{(2)})-\Delta
u+g(w^{(1)})-g(w^{(2)})=0\text{,}
\end{equation*}%
by $\frac{u(t)-u(t-h)}{2h}$ and $\frac{u(t+h)-u(t)}{2h}$ on $(\tau
+h,s+h)\times \Omega $ and $(\tau ,s)\times \Omega $, respectively, and then
summing these relations, we obtain%
\begin{equation*}
\frac{1}{2h}\left\langle u_{t}(s+h)+u_{t}(s),u(s+h)-u(s)\right\rangle -\frac{%
1}{2h}\left\langle u_{t}(\tau +h)+u_{t}(\tau ),u(\tau +h)-u(\tau
)\right\rangle
\end{equation*}%
\begin{equation*}
-\frac{1}{2h}\underset{\tau }{\overset{s}{\int }}\left\langle
u_{t}(t+h)+u_{t}(t),u_{t}(t+h)-u_{t}(t)\right\rangle dt
\end{equation*}%
\begin{equation*}
+\frac{1}{2h}\underset{\tau }{\overset{s}{\int }}\left\langle \nabla
u_{t}(t+h)+\nabla u_{t}(t),\nabla u(t+h)-\nabla u(t)\right\rangle dt
\end{equation*}%
\begin{equation*}
+\underset{\tau }{\overset{s}{\int }}\left\langle
f(w_{t}^{(1)}(t))-f(w_{t}^{(2)}(t)),\frac{u(t+h)-u(t-h)}{2h}\right\rangle dt
\end{equation*}%
\begin{equation*}
+\underset{s}{\overset{s+h}{\int }}\left\langle
f(w_{t}^{(1)}(t))-f(w_{t}^{(2)}(t)),\frac{u(t)-u(t-h)}{2h}\right\rangle dt
\end{equation*}%
\begin{equation*}
-\underset{\tau }{\overset{\tau +h}{\int }}\left\langle
f(w_{t}^{(1)}(t))-f(w_{t}^{(2)}(t)),\frac{u(t+h)-u(t)}{2h}\right\rangle dt
\end{equation*}%
\begin{equation*}
+\frac{1}{2h}\underset{\tau }{\overset{s}{\int }}\left\langle \nabla
u(t+h)+\nabla u(t),\nabla u(t+h)-\nabla u(t)\right\rangle dt
\end{equation*}%
\begin{equation*}
=\frac{1}{2}\underset{\tau }{\overset{s}{\int }}\left\langle
g(w^{(2)}(t+h))-g(w^{(1)}(t+h)),\frac{u(t+h)-u(t)}{h}\right\rangle dt
\end{equation*}%
\begin{equation}
+\frac{1}{2}\underset{\tau }{\overset{s}{\int }}\left\langle
g(w^{(2)}(t))-g(w^{(1)}(t)),\frac{u(t+h)-u(t)}{h}\right\rangle dt,\text{ \ }%
\forall \lbrack \tau ,s]\subset (0,T),  \tag{3.27}
\end{equation}%
where $h$ is a sufficiently small positive number. Since
\begin{equation*}
\frac{1}{2h}\left\langle u_{t}(s+h)+u_{t}(s),u(s+h)-u(s)\right\rangle -\frac{%
1}{2h}\underset{\tau }{\overset{s}{\int }}\left\langle
u_{t}(t+h)+u_{t}(t),u_{t}(t+h)-u_{t}(t)\right\rangle dt
\end{equation*}%
\begin{equation*}
+\frac{1}{2h}\underset{\tau }{\overset{s}{\int }}\left\langle \nabla
u(t+h)+\nabla u(t),\nabla u(t+h)-\nabla u(t)\right\rangle dt=\frac{1}{4h}%
\frac{d}{ds}\left\langle u(s+h)-u(s),u(s+h)-u(s)\right\rangle
\end{equation*}%
\begin{equation*}
+\frac{1}{h}\left\langle u_{t}(s),u(s+h)-u(s)\right\rangle -\frac{1}{2h}%
\underset{s}{\overset{s+h}{\int }}\left\Vert u_{t}(t)\right\Vert
_{L^{2}(\Omega )}^{2}dt+\frac{1}{2h}\underset{\tau }{\overset{\tau +h}{\int }%
}\left\Vert u_{t}(t)\right\Vert _{L^{2}(\Omega )}^{2}dt
\end{equation*}%
\begin{equation*}
+\frac{1}{2h}\underset{s}{\overset{s+h}{\int }}\left\Vert \nabla
u(t)\right\Vert _{L^{2}(\Omega )}^{2}dt-\frac{1}{2h}\underset{\tau }{\overset%
{\tau +h}{\int }}\left\Vert \nabla u(t)\right\Vert _{L^{2}(\Omega )}^{2}dt
\end{equation*}%
and
\begin{equation*}
\frac{1}{2}\underset{\tau }{\overset{s}{\int }}\left\langle
g(w^{(2)}(t+h))-g(w^{(1)}(t+h)),\frac{u(t+h)-u(t)}{h}\right\rangle dt
\end{equation*}%
\begin{equation*}
+\frac{1}{2}\underset{\tau }{\overset{s}{\int }}\left\langle
g(w^{(2)}(t))-g(w^{(1)}(t)),\frac{u(t+h)-u(t)}{h}\right\rangle dt
\end{equation*}%
\begin{equation*}
\leq c\underset{\tau }{\overset{s}{\int }}\left( \left\Vert
u(t+h)\right\Vert _{L^{2}(\Omega )}+\left\Vert u(t)\right\Vert
_{L^{2}(\Omega )}\right) \left\Vert \frac{u(t+h)-u(t)}{h}\right\Vert
_{L^{2}(\Omega )}dt\text{,}
\end{equation*}%
integrating (3.27) on $(\tau ,\sigma )$ with respect to $s$, passing to the
limit as $h\searrow 0$ and taking into account Lemma A.4-A.5, we get%
\begin{equation*}
\underset{\tau }{\overset{\sigma }{\int }}\left\Vert u_{t}(s)\right\Vert
_{L^{2}(\Omega )}^{2}ds+\frac{1}{2}\underset{\tau }{\overset{\sigma }{\int }}%
\left\Vert \nabla u(s)\right\Vert _{L^{2}(\Omega )}^{2}ds+\underset{\tau }{%
\overset{\sigma }{\int }}\left\Vert \nabla u_{t}(s)\right\Vert
_{L^{2}(\Omega )}^{2}ds
\end{equation*}%
\begin{equation*}
\leq \frac{1}{2}\underset{h\searrow 0}{\lim \sup }\frac{1}{h}\underset{\tau }%
{\overset{\sigma }{\int }}\underset{s}{\overset{s+h}{\int }}\left\Vert
u_{t}(t)\right\Vert _{L^{2}(\Omega )}^{2}dtds+(\sigma -\tau )\left\Vert
\nabla u(\tau )\right\Vert _{L^{2}(\Omega )}^{2}+\widehat{c}(\sigma -\tau
)\left\Vert u_{t}(\tau )\right\Vert _{L^{2}(\Omega )}
\end{equation*}%
\begin{equation}
+2c\underset{0}{\overset{\sigma }{\int }}\underset{0}{\overset{s}{\int }}%
\left\Vert u(t)\right\Vert _{L^{2}(\Omega )}\left\Vert u_{t}(t)\right\Vert
_{L^{2}(\Omega )}dtds,\text{ \ \ }\forall \lbrack \tau ,\sigma ]\subset (0,T)%
\text{.}  \tag{3.28}
\end{equation}%
By Lebesgue's convergence theorem, we find%
\begin{equation*}
\underset{h\searrow 0}{\lim }\frac{1}{h}\underset{\tau }{\overset{\sigma }{%
\int }}\underset{s}{\overset{s+h}{\int }}\left\Vert u_{t}(t)\right\Vert
_{L^{2}(\Omega )}^{2}dtds=\underset{h\searrow 0}{\lim }\frac{1}{h}\underset{%
\tau }{\overset{\sigma }{\int }}\underset{0}{\overset{h}{\int }}\left\Vert
u_{t}(s+t)\right\Vert _{L^{2}(\Omega )}^{2}dtds=\underset{h\searrow 0}{\lim }%
\underset{\tau }{\overset{\sigma }{\int }}\underset{0}{\overset{1}{\int }}%
\left\Vert u_{t}(s+h\lambda )\right\Vert _{L^{2}(\Omega )}^{2}d\lambda ds
\end{equation*}%
\begin{equation*}
=\underset{h\searrow 0}{\lim }\underset{0}{\overset{1}{\int }}\underset{\tau
}{\overset{\sigma }{\int }}\left\Vert u_{t}(s+h\lambda )\right\Vert
_{L^{2}(\Omega )}^{2}dsd\lambda =\underset{h\searrow 0}{\lim }\underset{0}{%
\overset{1}{\int }}\text{ }\underset{\tau +h\lambda }{\overset{\sigma
+h\lambda }{\int }}\left\Vert u_{t}(\tau )\right\Vert _{L^{2}(\Omega
)}^{2}d\tau d\lambda =\underset{\tau }{\overset{\sigma }{\int }}\left\Vert
u_{t}(\tau )\right\Vert _{L^{2}(\Omega )}^{2}d\tau \text{.}
\end{equation*}%
Hence, taking the last equality into consideration in (3.28) and then
passing to the limit as $\tau \searrow 0$, we have%
\begin{equation*}
\underset{0}{\overset{t}{\int }}\left( \left\Vert u_{t}(s)\right\Vert
_{L^{2}(\Omega )}^{2}+\left\Vert \nabla u(s)\right\Vert _{L^{2}(\Omega
)}^{2}\right) ds\leq 2c\underset{0}{\overset{t}{\int }}\underset{0}{\overset{%
s}{\int }}\left( \left\Vert u(\sigma )\right\Vert _{L^{2}(\Omega
)}^{2}+\left\Vert u_{t}(\sigma )\right\Vert _{L^{2}(\Omega )}^{2}\right)
d\sigma ds,\text{ \ }\forall t\in \lbrack 0,T]\text{.}
\end{equation*}%
Denoting $y(s)=\underset{0}{\overset{s}{\int }}\left( \left\Vert u(\sigma
)\right\Vert _{L^{2}(\Omega )}^{2}+\left\Vert u_{t}(\sigma )\right\Vert
_{L^{2}(\Omega )}^{2}\right) d\sigma $, by the last inequality, it follows
that%
\begin{equation*}
y(t)\leq 2c\underset{0}{\overset{t}{\int }}y(s)ds,\text{ \ }\forall t\in
\lbrack 0,T]\text{.}
\end{equation*}%
Thus, applying Gronwall's lemma, we obtain $y(s)=0$ and consequently $%
u(s,.)=0$, for every $s\in \lbrack 0,T]$.
\end{proof}

Denoting $H(t,x)=-g(w(t,x)+h(x)$, by (3.26), we have $H\in L^{\infty
}(0,T;L^{2}(\Omega ))$. Also, due to (2.1),%
\begin{equation}
\left\{
\begin{array}{c}
w_{tt}-\Delta w_{t}+f(w_{t})-\Delta w=H(t,x)\text{\ \ \ \ in \ }(0,T)\times
\Omega \text{, \ } \\
w=0\text{ \ \ \ \ \ \ \ \ \ \ \ \ \ \ \ \ \ \ \ \ \ \ \ \ \ \ \ \ \ \ \ \ \
\ \ \ \ \ \ \ on \ }(0,T)\times \partial \Omega \text{,} \\
w(0,\cdot )=w_{0}\text{ },\text{ \ \ \ \ \ \ \ \ \ \ }w_{t}(0,\cdot )=w_{1}%
\text{\ \ \ \ \ \ \ \ \ \ \ \ \ \ \ \ \ in \ }\Omega \text{.}%
\end{array}%
\right.  \tag{3.29}
\end{equation}%
Applying the techniques of the proof of Lemma 3.3, we can say that the
function $w\in C(\left[ 0,T\right] ;H_{0}^{1}(\Omega )\cap L^{\infty
}(\Omega ))$ with $w_{t}\in C_{s}(0,T;L^{2}(\Omega ))\cap
L^{2}(0,T;H_{0}^{1}(\Omega ))$, is a unique solution to (3.29).

Formally multiplying (3.29) by $w_{tt}$ and integrating over $(s,T)\times
\Omega $, we obtain%
\begin{equation*}
\frac{1}{2}\underset{s}{\overset{T}{\int }}\left\Vert w_{tt}(t)\right\Vert
_{L^{2}(\Omega )}^{2}dt+\frac{1}{2}\left\Vert \nabla w_{t}(T)\right\Vert
_{L^{2}(\Omega )}^{2}+\left\langle F(w(T)),1\right\rangle +\left\langle
\nabla w(T),\nabla w_{t}(T)\right\rangle
\end{equation*}%
\begin{equation*}
\leq \underset{s}{\overset{T}{\int }}\left\Vert \nabla w_{t}(t)\right\Vert
_{L^{2}(\Omega )}^{2}dt+\frac{1}{2}\left\Vert \nabla w_{t}(s)\right\Vert
_{L^{2}(\Omega )}^{2}+\left\langle F(w(s)),1\right\rangle
\end{equation*}%
\begin{equation*}
+\left\langle \nabla w(s),\nabla w_{t}(s)\right\rangle +\frac{1}{2}\underset{%
s}{\overset{T}{\int }}\left\Vert H(t)\right\Vert _{L^{2}(\Omega )}^{2}dt%
\text{,}
\end{equation*}%
where $F(w)=\underset{0}{\overset{w}{\int }}f(s)ds$. Integrating the last
equality on $[0,T]$ with respect to $s$ and taking into account (2.2) and
(2.5), we find
\begin{equation}
\underset{0}{\overset{T}{\int }}\underset{s}{\overset{T}{\int }}\left\Vert
w_{tt}(t)\right\Vert _{L^{2}(\Omega )}^{2}dtds\leq \widetilde{c}%
(T,\left\Vert (w_{0},w_{1})\right\Vert _{(H_{0}^{1}(\Omega )\cap L^{\infty
}(\Omega ))\times L^{2}(\Omega )})\text{.}  \tag{3.30}
\end{equation}%
Since the problem (3.29) admits a unique solution, using Galerkin's
approximations one can justify (3.30). So, by (3.30), we have $w_{tt}\in
L_{loc}^{2}(0,T;L^{2}(\Omega ))$, which together with $w_{t}\in
C_{s}(0,T;L^{2}(\Omega ))$ and (3.1) implies that $w_{t}\in C(\left[ 0,T%
\right] ;L^{2}(\Omega ))$.

Thus, to finish the proof of Theorem 2.1, we just need to show (2.6). Let $w$%
, $v\in $ $C([0,T];H_{0}^{1}(\Omega )\cap L^{\infty }(\Omega ))\cap
C^{1}([0,T];L^{2}(\Omega ))\cap W^{1,2}(0,T;H_{0}^{1}(\Omega ))\cap
W_{loc}^{2,2}(0,T;L^{2}(\Omega ))$ are the weak solutions to (2.1). Then due
to (2.1)$_{1},$ the function $u(t,x)=w(t,x)-v(t,x)$ satisfies the equation%
\begin{equation*}
u_{tt}-\Delta u_{t}+f(w_{t})-f(v_{t})-\Delta u=g(v)-g(w)\text{.}
\end{equation*}%
Testing the above equation by $2u_{t}$ on $(s,t)\times \Omega $ and taking
into account (2.2) and (2.5), we get%
\begin{equation*}
\left\Vert u_{t}(t)\right\Vert _{L^{2}(\Omega )}^{2}+\left\Vert \nabla
u(t)\right\Vert _{L^{2}(\Omega )}^{2}\leq M\underset{s}{\overset{t}{\int }}%
\left( \left\Vert u_{t}(\tau )\right\Vert _{L^{2}(\Omega )}^{2}+\left\Vert
\nabla u(\tau )\right\Vert _{L^{2}(\Omega )}^{2}\right) d\tau
\end{equation*}%
\begin{equation*}
+\left\Vert u_{t}(s)\right\Vert _{L^{2}(\Omega )}^{2}+\left\Vert \nabla
u(s)\right\Vert _{L^{2}(\Omega )}^{2},\text{ \ \ }0<s\leq t<T\text{.}
\end{equation*}%
Therefore, applying Gronwall's lemma, we obtain (2.6).

\bigskip

\section{Dissipativity and asymptotic compactness}

We begin with the following dissipativity result.

\begin{lemma}
Assume that the conditions (2.2)-(2.4) are satisfied. Then for every bounded
subset $B\subset (H_{0}^{1}(\Omega )\cap L^{\infty }(\Omega ))\times
L^{2}(\Omega )$, there exists $c_{B}>0$ such that%
\begin{equation}
\underset{\varphi \in B}{\sup }\left\Vert S(t)\varphi \right\Vert
_{(H_{0}^{1}(\Omega )\cap L^{\infty }(\Omega ))\times L^{2}(\Omega )}\leq
c_{B}\text{, \ \ }\forall t\geq 0\text{.}  \tag{4.1}
\end{equation}
\end{lemma}

\begin{proof}
Let $\varphi \in B$ and $(w(t),w_{t}(t))=S(t)\varphi $. By (2.5),
immediately, it follows that%
\begin{equation}
\left\Vert S(t)\varphi \right\Vert _{H_{0}^{1}(\Omega )\times L^{2}(\Omega
)}\leq c_{B}^{(1)},\text{ \ \ }\forall t\geq 0\text{.}  \tag{4.2}
\end{equation}

Denote by $v^{(s)}(t,x)$ the weak solution of (3.15) with $(w(s),w_{t}(s))$
and $w(t+s)$ instead of $(w_{0},w_{1})$ and $w(t)$, respectively. Also,
denote $u^{(s)}(t,x)=w(t+s,x)-v^{(s)}(t,x)$, for $(t,x)\in \lbrack
0,1]\times \Omega $ and $s\geq 0$. Then by Lemma 3.2-3.3 and (2.5), we find%
\begin{equation*}
\left\Vert w(t+s)\right\Vert _{L^{\infty }(\Omega )}\leq e^{-t}\left\Vert
w(s)\right\Vert _{L^{\infty }(\Omega )}+c_{B}^{(2)},\text{ \ }\forall t\in
\lbrack 0,1]\text{ and }\forall s\geq 0\text{ .}
\end{equation*}%
By the iteration, it follows that%
\begin{equation*}
\left\Vert w(n)\right\Vert _{L^{\infty }(\Omega )}\leq e^{-n}\left\Vert
w(0)\right\Vert _{L^{\infty }(\Omega )}+c_{B}^{(2)}\frac{1-e^{-(n-1)}}{%
1-e^{-1}}
\end{equation*}%
and consequently
\begin{equation}
\left\Vert w(n)\right\Vert _{L^{\infty }(\Omega )}\leq \left\Vert
w(0)\right\Vert _{L^{\infty }(\Omega )}+\frac{e}{e-1}c_{B}^{(2)},\text{ \ }%
\forall n\in
\mathbb{Z}
_{+}\text{.}  \tag{4.3}
\end{equation}%
Since for every $T\geq 0$ there exist $n_{T}\in
\mathbb{Z}
_{+}$ and $t_{T}\in \lbrack 0,1)$ such that%
\begin{equation*}
T=n_{T}+t_{T}\text{,}
\end{equation*}%
by (4.3), we get
\begin{equation*}
\left\Vert w(T)\right\Vert _{L^{\infty }(\Omega )}\leq c_{B}^{(3)},\text{ \ }%
\forall T\geq 0\text{,\ }
\end{equation*}%
which, together with (4.2), yields (4.1).
\end{proof}

Now, we will show asymptotic compactness of $\left\{ S(t)\right\} _{t\geq
0}^{\infty }$ in $(H_{0}^{1}(\Omega )\cap L^{\infty }(\Omega ))\times
L^{2}(\Omega )$.

\begin{lemma}
Let the conditions (2.2)-(2.4) be satisfied and $B$ be bounded subset of $%
(H_{0}^{1}(\Omega )\cap L^{\infty }(\Omega ))\times L^{2}(\Omega )$. Then
the set $\left\{ S(t_{m})\varphi _{m}\right\} _{m=1}^{\infty }$ is
relatively compact in $(H_{0}^{1}(\Omega )\cap L^{\infty }(\Omega ))\times
L^{2}(\Omega )$, where $t_{m}\rightarrow \infty $ and $\left\{ \varphi
_{m}\right\} _{m=1}^{\infty }\subset B$.
\end{lemma}

\begin{proof}
\textit{Step1}. We first prove relatively compactness of $\left\{
S(t_{m})\varphi _{m}\right\} _{m=1}^{\infty }$ in $H_{0}^{1}(\Omega )\times
L^{2}(\Omega )$. For any $T>0$ and $m\in
\mathbb{N}
$ such that $t_{m}\geq T$, let us define%
\begin{equation*}
(w_{m}^{(T)}(t),w_{mt}^{(T)}(t))=S(t+t_{m}-T)\varphi _{m}\text{.}
\end{equation*}%
By Lemma 4.1 and (2.5)$_{1}$, it follows that%
\begin{equation*}
\left\Vert w_{m}^{(T)}(t)\right\Vert _{H_{0}^{1}(\Omega )\cap L^{\infty
}(\Omega )}+\left\Vert w_{mt}^{(T)}(t)\right\Vert _{L^{2}(\Omega )}
\end{equation*}%
\begin{equation}
+\underset{0}{\overset{t}{\int }}\left\Vert \nabla
w_{mt}^{(T)}(s)\right\Vert _{L^{2}(\Omega )}^{2}ds+\underset{0}{\overset{t}{%
\int }}\left\langle f_{1}(w_{mt}^{(T)}(s)),w_{mt}^{(T)}(s)\right\rangle
ds\leq c_{1},\text{ }\forall t\geq 0\text{.}  \tag{4.4}
\end{equation}%
Then, by Banach-Alaoglu theorem, there exist a subsequence $\left\{
w_{m_{n}}^{(T)}(t)\right\} _{n=1}^{\infty }$ and a function $w\in L^{\infty
}(0,T;H_{0}^{1}(\Omega )\cap L^{\infty }(\Omega ))\cap W^{1,\infty
}(0,T;L^{2}(\Omega ))\cap W^{1,2}(0,T;H_{0}^{1}(\Omega ))$ such that%
\begin{equation}
\left\{
\begin{array}{c}
w_{m_{n}}^{(T)}(t)\rightarrow w(t)\text{ weakly star in }L^{\infty
}(0,T;H_{0}^{1}(\Omega )\cap L^{\infty }(\Omega ))\text{,} \\
w_{m_{n}t}^{(T)}(t)\rightarrow w_{t}(t)\text{ weakly star in }L^{\infty
}(0,T;L^{2}(\Omega ))\text{, \ \ \ \ \ \ \ \ \ \ } \\
w_{m_{n}t}^{(T)}(t)\rightarrow w_{t}(t)\text{ weakly in }%
L^{2}(0,T;H_{0}^{1}(\Omega ))\text{.\ \ \ \ \ \ \ \ \ \ \ \ \ \ \ }%
\end{array}%
\right.  \tag{4.5}
\end{equation}%
Testing the equation%
\begin{equation*}
w_{m_{n}tt}^{(T)}-w_{m_{k}tt}^{(T)}-\Delta
(w_{m_{n}t}^{(T)}-w_{m_{k}t}^{(T)})-\Delta
(w_{m_{n}}^{(T)}-w_{m_{k}}^{(T)})+f(w_{m_{n}t}^{(T)})-f(w_{m_{k}t}^{(T)})
\end{equation*}%
\begin{equation}
+g(w_{m_{n}}^{(T)})-g(w_{m_{k}}^{(T)})=0  \tag{4.6}
\end{equation}%
by $2(w_{m_{n}t}^{(T)}-w_{m_{k}t}^{(T)})$ on $(s,T)\times \Omega $ and
taking into account (2.2) and (4.4), we find%
\begin{equation*}
\left\Vert w_{m_{n}t}^{(T)}(T)-w_{m_{k}t}^{(T)}(T)\right\Vert _{L^{2}(\Omega
)}^{2}+\left\Vert \nabla (w_{m_{n}}^{(T)}(T)-w_{m_{k}}^{(T)}(T))\right\Vert
_{L^{2}(\Omega )}^{2}
\end{equation*}%
\begin{equation*}
+c_{1}\underset{s}{\overset{T}{\int }}\left\Vert
w_{m_{n}t}^{(T)}(t)-w_{m_{k}t}^{(T)}(t)\right\Vert _{L^{2}(\Omega
)}^{2}dt\leq \left\Vert w_{m_{n}t}^{(T)}(s)-w_{m_{k}t}^{(T)}(s)\right\Vert
_{L^{2}(\Omega )}^{2}+\left\Vert \nabla
(w_{m_{n}}^{(T)}(s)-w_{m_{k}}^{(T)}(s))\right\Vert _{L^{2}(\Omega )}^{2}
\end{equation*}%
\begin{equation}
+c_{2}\underset{s}{\overset{T}{\int }}\left\Vert
w_{m_{n}}^{(T)}(t)-w_{m_{k}}^{(T)}(t)\right\Vert _{L^{2}(\Omega )}^{2}dt%
\text{.}  \tag{4.7}
\end{equation}%
Also, testing (4.6) by $(w_{m_{n}}^{(T)}-w_{m_{k}}^{(T)})$ on $(0,T)\times
\Omega $ and considering (4.4), we get%
\begin{equation*}
\underset{0}{\overset{T}{\int }}\left\Vert \nabla
(w_{m_{n}}^{(T)}(t)-w_{m_{k}}^{(T)}(t))\right\Vert _{L^{2}(\Omega
)}^{2}dt\leq c_{3}+\underset{0}{\overset{T}{\int }}\left\Vert
w_{m_{n}t}^{(T)}(t)-w_{m_{k}t}^{(T)}(t)\right\Vert _{L^{2}(\Omega )}^{2}dt
\end{equation*}%
\begin{equation*}
+c_{4}\underset{0}{\overset{T}{\int }}\left\Vert
w_{m_{n}}^{(T)}(t)-w_{m_{k}}^{(T)}(t)\right\Vert _{L^{2}(\Omega )}^{2}dt
\end{equation*}%
\begin{equation*}
+\left\vert \underset{0}{\overset{T}{\int }}\left\langle
f(w_{m_{n}t}^{(T)}(t))-f(w_{m_{k}t}^{(T)}(t)),w_{m_{n}}^{(T)}(t)-w_{m_{k}}^{(T)}(t)\right\rangle dt\right\vert
\text{,}
\end{equation*}%
which, together with (4.7), implies that%
\begin{equation*}
\underset{0}{\overset{T}{\int }}\left\Vert
w_{m_{n}t}^{(T)}(t)-w_{m_{k}t}^{(T)}(t)\right\Vert _{L^{2}(\Omega )}^{2}dt+%
\underset{0}{\overset{T}{\int }}\left\Vert \nabla
(w_{m_{n}}^{(T)}(t)-w_{m_{k}}^{(T)}(t))\right\Vert _{L^{2}(\Omega )}^{2}dt
\end{equation*}%
\begin{equation*}
\leq c_{5}+c_{6}\underset{0}{\overset{T}{\int }}\left\Vert
w_{m_{n}}^{(T)}(t)-w_{m_{k}}^{(T)}(t)\right\Vert _{L^{2}(\Omega )}^{2}dt+
\end{equation*}%
\begin{equation}
+\left\vert \underset{0}{\overset{T}{\int }}\left\langle
f(w_{m_{n}t}^{(T)}(t))-f(w_{m_{k}t}^{(T)}(t)),w_{m_{n}}^{(T)}(t)-w_{m_{k}}^{(T)}(t)\right\rangle dt\right\vert
\text{.}  \tag{4.8}
\end{equation}%
Integrating (4.7) over $(0,T)$ with respect to $s$ and taking (4.8) into
consideration, we obtain%
\begin{equation*}
T\left\Vert w_{m_{n}t}^{(T)}(T)-w_{m_{k}t}^{(T)}(T)\right\Vert
_{L^{2}(\Omega )}^{2}+T\left\Vert \nabla
(w_{m_{n}}^{(T)}(T)-w_{m_{k}}^{(T)}(T))\right\Vert _{L^{2}(\Omega )}^{2}
\end{equation*}%
\begin{equation*}
\leq c_{5}+c_{7}(1+T)\underset{0}{\overset{T}{\int }}\left\Vert
w_{m_{n}}^{(T)}(t)-w_{m_{k}}^{(T)}(t)\right\Vert _{L^{2}(\Omega )}^{2}dt
\end{equation*}%
\begin{equation}
+\left\vert \underset{0}{\overset{T}{\int }}\left\langle
f(w_{m_{n}t}^{(T)}(t))-f(w_{m_{k}t}^{(T)}(t)),w_{m_{n}}^{(T)}(t)-w_{m_{k}}^{(T)}(t)\right\rangle dt\right\vert
\text{.}  \tag{4.9}
\end{equation}%
By using the compact embedding theorem (see \cite{22}), from (4.5), it
follows that
\begin{equation*}
\underset{n,k\rightarrow \infty }{\lim }\underset{0}{\overset{T}{\int }}%
\left\Vert w_{m_{n}}^{(T)}(t)-w_{m_{k}}^{(T)}(t)\right\Vert _{L^{2}(\Omega
)}^{2}dt=0
\end{equation*}%
and taking into account (4.4), we have
\begin{equation*}
\underset{n\rightarrow \infty }{\lim \sup }\underset{k\rightarrow \infty }{%
\lim \sup }\left\vert \underset{0}{\overset{T}{\int }}\left\langle
f(w_{m_{n}t}^{(T)}(t))-f(w_{m_{k}t}^{(T)}(t)),w_{m_{n}}^{(T)}(t)-w_{m_{k}}^{(T)}(t)\right\rangle dt\right\vert
\end{equation*}%
\begin{equation*}
\leq \underset{n\rightarrow \infty }{\lim \sup }\underset{k\rightarrow
\infty }{\lim \sup }\left\vert \underset{0}{\overset{T}{\int }}\underset{%
\left\{ x:x\in \Omega \text{ }\left\vert w_{m_{n}t}(t,x)\right\vert
>1\right\} }{\int }\left\langle
f(w_{m_{n}t}^{(T)}(t)),w_{m_{n}}^{(T)}(t)-w_{m_{k}}^{(T)}(t)\right\rangle
dt\right\vert
\end{equation*}%
\begin{equation*}
+\underset{n\rightarrow \infty }{\lim \sup }\underset{k\rightarrow \infty }{%
\lim \sup }\left\vert \underset{0}{\overset{T}{\int }}\underset{\left\{
x:x\in \Omega \text{ }\left\vert w_{m_{k}t}(t,x)\right\vert >1\right\} }{%
\int }\left\langle
f(w_{m_{k}t}^{(T)}(t)),w_{m_{n}}^{(T)}(t)-w_{m_{k}}^{(T)}(t)\right\rangle
dt\right\vert
\end{equation*}%
\begin{equation*}
\leq c_{8}\underset{n\rightarrow \infty }{\lim \sup }\underset{0}{\overset{T}%
{\int }}\left\langle
f_{1}(w_{m_{n}t}^{(T)}(t)),w_{m_{n}t}^{(T)}(t)\right\rangle dt+c_{8}\underset%
{n\rightarrow \infty }{\lim \sup }\underset{0}{\overset{T}{\int }}\left\Vert
w_{m_{n}t}^{(T)}(t)\right\Vert _{L^{2}(\Omega )}^{2}dt\leq c_{9}\text{.}
\end{equation*}%
Hence, passing to the limit in (4.9), we obtain%
\begin{equation*}
\underset{n\rightarrow \infty }{\lim \sup }\underset{k\rightarrow \infty }{%
\lim \sup }\left\Vert S(t_{m_{n}})\varphi _{m_{n}}-S(t_{m_{k}})\varphi
_{m_{k}}\right\Vert _{H_{0}^{1}(\Omega )\times L^{2}(\Omega )}\leq \frac{%
c_{10}}{\sqrt{T}},\text{ \ \ }\forall T>0
\end{equation*}%
and consequently%
\begin{equation}
\underset{n\rightarrow \infty }{\lim }\underset{k\rightarrow \infty }{\lim
\sup }\left\Vert S(t_{m_{n}})\varphi _{m_{n}}-S(t_{m_{k}})\varphi
_{m_{k}}\right\Vert _{H_{0}^{1}(\Omega )\times L^{2}(\Omega )}=0\text{.}
\tag{4.10}
\end{equation}%
So, passing to the limit in the inequality%
\begin{equation*}
\left\Vert S(t_{m_{n}})\varphi _{m_{n}}-S(t_{m_{\nu }})\varphi _{m_{\nu
}}\right\Vert _{H_{0}^{1}(\Omega )\times L^{2}(\Omega )}
\end{equation*}%
\begin{equation*}
\leq \underset{k\rightarrow \infty }{\lim \sup }\left\Vert
S(t_{m_{n}})\varphi _{m_{n}}-S(t_{m_{k}})\varphi _{m_{k}}\right\Vert
_{H_{0}^{1}(\Omega )\times L^{2}(\Omega )}+\underset{k\rightarrow \infty }{%
\lim \sup }\left\Vert S(t_{m_{k}})\varphi _{m_{k}}-S(t_{m_{\nu }})\varphi
_{m_{\nu }}\right\Vert _{H_{0}^{1}(\Omega )\times L^{2}(\Omega )}
\end{equation*}%
and taking (4.10) into consideration, we get%
\begin{equation*}
\underset{n,\nu \rightarrow \infty }{\lim }\left\Vert S(t_{m_{n}})\varphi
_{m_{n}}-S(t_{m_{\nu }})\varphi _{m_{\nu }}\right\Vert _{H_{0}^{1}(\Omega
)\times L^{2}(\Omega )}=0.
\end{equation*}%
Thus, the subsequence $\left\{ S(t_{m_{n}})\varphi _{m_{n}}\right\}
_{n=1}^{\infty }$ is a Cauchy sequence in $H_{0}^{1}(\Omega )\times
L^{2}(\Omega )$ and consequently converges. By the same way, we can show
that every subsequence of $\left\{ S(t_{m})\varphi _{m}\right\}
_{m=1}^{\infty }$ has a convergent subsequence in $H_{0}^{1}(\Omega )\times
L^{2}(\Omega )$. This gives us relatively compactness of $\left\{
S(t_{m})\varphi _{m}\right\} _{m=1}^{\infty }$ in $H_{0}^{1}(\Omega )\times
L^{2}(\Omega )$.

\textit{Step2. }Now, let us prove relatively compactness of $\left\{
w_{m}(t_{m})\right\} _{m=1}^{\infty }$ in $L^{\infty }(\Omega )$, where $%
(w_{m}(t),w_{mt}(t)=S(t)\varphi _{m}$. By the relatively compactness of $%
\left\{ S(t_{m})\varphi _{m}\right\} _{m=1}^{\infty }$ in $H_{0}^{1}(\Omega
)\times L^{2}(\Omega )$, for any $\delta >0$ there exist $T_{\delta }>0$ and
$M_{\delta }>0$ such that%
\begin{equation}
\underset{\left\{ x:x\in \Omega \text{ }\left\vert w_{mt}(t,x)\right\vert
>M_{\delta }\right\} }{\int }\left( \left\vert \nabla w_{m}(t,x)\right\vert
^{2}+\left\vert w_{mt}(t,x)\right\vert ^{2}\right) dx<\delta ,\text{ \ }%
\forall t\geq T_{\delta }\text{.}  \tag{4.11}
\end{equation}%
Testing the equation%
\begin{equation*}
w_{mtt}-\Delta w_{mt}+f(w_{mt})-\Delta w_{m}+g(w_{m})=h
\end{equation*}%
by $\left\{
\begin{array}{c}
w_{mt}+M_{\delta },\text{ \ \ }w_{mt}\leq -M_{\delta }, \\
0,\text{ \ \ \ \ \ \ \ \ \ \ \ \ \ \ }\left\vert w_{mt}\right\vert \leq
M_{\delta }, \\
w_{mt}-M_{\delta },\text{ \ \ \ }w_{mt}\geq M_{\delta }\text{ \ }%
\end{array}%
\right. $ on $(s,s+1)\times \Omega $ and taking into account (4.11), we find%
\begin{equation*}
\underset{s}{\overset{s+1}{\int }}\underset{\left\{ x:x\in \Omega \text{ }%
w_{mt}(t,x)>M_{\delta }\right\} }{\int }f_{1}(w_{mt}(t,x))(w_{mt}(t,x)-M_{%
\delta })dxdt
\end{equation*}%
\begin{equation*}
+\underset{s}{\overset{s+1}{\int }}\underset{\left\{ x:x\in \Omega \text{ }%
w_{mt}(t,x)<-M_{\delta }\right\} }{\int }f_{1}(w_{mt}(t,x))(w_{mt}(t,x)+M_{%
\delta })dxdt\leq c_{11}\delta ,\text{ \ }\forall s\geq T_{\delta }\text{,}
\end{equation*}%
and consequently%
\begin{equation}
\underset{s}{\overset{s+1}{\int }}\underset{\left\{ x:x\in \Omega \text{ }%
\left\vert w_{mt}(t,x)\right\vert >2M_{\delta }\right\} }{\int }%
f_{1}(w_{mt}(t,x))w_{mt}(t,x)dxdt\leq 2c_{11}\delta ,\text{ \ }\forall s\geq
T_{\delta }\text{.}  \tag{4.12}
\end{equation}%
Now, denote $w_{m}^{\tau }(t,x)$ $=w_{m}(t+\tau ,x)$, $f_{1}^{M}(y)=\left\{
\begin{array}{c}
f_{1}(w),\text{ \ \ }\left\vert y\right\vert \leq M \\
0,\text{ \ \ \ \ \ \ }\left\vert y\right\vert >M%
\end{array}%
\right. $, $\Gamma _{\varepsilon }(t,x)=(1+\varepsilon +\lambda
_{1})w_{mt}^{\tau }(t,x)+w_{m}^{\tau }(t,x)-g(w_{m}^{\tau
}(t,x))-f_{1}^{M}(w_{mt}^{\tau }(t,x))$\ $+h(x)$ and $\Phi
(y)=f_{1}(y)-f_{1}^{M}(y)+\varepsilon y$, where $M\geq 1$ and $\varepsilon
\in (0,1)$. Decompose $w_{m}^{\tau }(t)$ as $w_{m}^{\tau }(t)=v_{\varepsilon
m}^{\tau }(t)+u_{\varepsilon m}^{\tau }(t)$, where $v_{\varepsilon m}^{\tau
}(t)$ and $u_{\varepsilon m}^{\tau }(t)$ are solutions of the problems%
\begin{equation*}
\left\{
\begin{array}{c}
v_{\varepsilon mtt}^{\tau }-\Delta v_{\varepsilon mt}^{\tau }+v_{\varepsilon
mt}^{\tau }-\Delta v_{\varepsilon m}^{\tau }=\Gamma _{\varepsilon }(t,x)%
\text{\ \ \ \ in \ }(0,1)\times \Omega \text{,} \\
v_{\varepsilon m}^{\tau }=0\text{ \ \ \ \ \ \ \ \ \ \ \ \ \ \ \ \ \ \ \ \ \
\ \ \ \ \ \ \ \ \ \ \ \ \ \ \ \ \ \ \ \ \ \ on \ }(0,1)\times \partial
\Omega \text{,} \\
v_{\varepsilon m}^{\tau }(0,\cdot )=w_{m}^{\tau }(0),\text{ \ \ \ \ }%
v_{\varepsilon mt}^{\tau }(0,\cdot )=w_{mt}^{\tau }(0)\text{\ \ \ \ \ \ \ \
\ \ \ \ in \ }\Omega \text{,}%
\end{array}%
\right.
\end{equation*}%
and%
\begin{equation*}
\left\{
\begin{array}{c}
u_{\varepsilon mtt}^{\tau }-\Delta u_{\varepsilon mt}^{\tau }+u_{\varepsilon
mt}^{\tau }-\Delta u_{\varepsilon m}^{\tau }=-\Phi _{\varepsilon
}(w_{mt}^{\tau })\text{\ \ \ in \ }(0,1)\times \Omega \text{,} \\
u_{\varepsilon m}^{\tau }=0\text{ \ \ \ \ \ \ \ \ \ \ \ \ \ \ \ \ \ \ \ \ \
\ \ \ \ \ \ \ \ \ \ \ \ \ \ \ \ \ \ \ \ \ \ \ \ \ on \ }(0,1)\times \partial
\Omega \text{,} \\
u_{\varepsilon m}^{\tau }(0,\cdot )=0,\text{ \ \ \ \ \ \ \ \ \ \ }%
u_{\varepsilon mt}^{\tau }(0,\cdot )=0\text{\ \ \ \ \ \ \ \ \ \ \ \ \ \ \ \
\ \ in \ }\Omega \text{. \ \ \ \ \ }%
\end{array}%
\right.
\end{equation*}%
By using the techniques of the proof of Lemma 3.2, one can show that%
\begin{equation}
\left\Vert v_{\varepsilon m}^{\tau }(1)-e^{-1}w_{m}^{\tau }(0)\right\Vert
_{H^{s}(\Omega )}\leq \frac{\widehat{c}_{1}}{2-s}(1+\left\Vert
f_{1}\right\Vert _{C[-M,M]}),\text{ \ \ }\forall s\in \lbrack 0,2)\text{.}
\tag{4.13}
\end{equation}%
Also, using the arguments of the proof of Lemma 3.3, we have%
\begin{equation*}
\left\Vert u_{\varepsilon m}^{\tau }(t)\right\Vert _{L^{\infty }(\Omega
)}\leq \frac{2\varepsilon }{(2-\alpha )(1-\alpha )}
\end{equation*}%
\begin{equation}
+\frac{1}{4\pi \alpha }\left( \underset{\Phi ^{-1}(\varepsilon )}{\overset{%
\infty }{\int }}\frac{\Phi ^{\prime }(\nu )}{\nu \Phi (\nu )}d\nu +\underset{%
-\infty }{\overset{\Phi ^{-1}(-\varepsilon )}{\int }}\frac{\Phi ^{\prime
}(\nu )}{\nu \Phi (\nu )}d\nu \right) \underset{0}{\overset{1}{\int }}%
\left\langle \Phi (w_{mt}^{\tau }(s)),w_{mt}^{\tau }(s)\right\rangle ds,%
\text{ \ }\forall t\in \lbrack 0,1]\text{,}  \tag{4.14}
\end{equation}%
where $\alpha \in (0,1)$. By the definition of $\Phi $, it follows that%
\begin{equation*}
\Phi ^{-1}(\varepsilon )=1\text{ and }\Phi ^{-1}(-\varepsilon )=-1\text{.}
\end{equation*}%
Hence,%
\begin{equation*}
\underset{\Phi ^{-1}(\varepsilon )}{\overset{\infty }{\int }}\frac{\Phi
^{\prime }(\nu )}{\nu \Phi (\nu )}d\nu +\underset{-\infty }{\overset{\Phi
^{-1}(-\varepsilon )}{\int }}\frac{\Phi ^{\prime }(\nu )}{\nu \Phi (\nu )}%
d\nu =2\underset{1}{\overset{M}{\int }}\frac{1}{v^{2}}dv+\underset{M}{%
\overset{\infty }{\int }}\frac{f_{1}^{\prime }(v)+\varepsilon }{v\left(
f_{1}(v)+\varepsilon v\right) }dv
\end{equation*}%
\begin{equation*}
+\underset{-\infty }{\overset{-M}{\int }}\frac{f_{1}^{\prime
}(v)+\varepsilon }{v\left( f_{1}(v)+\varepsilon v\right) }dv\leq 2\underset{1%
}{\overset{\infty }{\int }}\frac{1}{v^{2}}dv+\underset{1}{\overset{\infty }{%
\int }}\frac{f_{1}^{\prime }(v)}{vf_{1}(v)}dv+\underset{-\infty }{\overset{-1%
}{\int }}\frac{f_{1}^{\prime }(v)}{vf_{1}(v)}dv\leq \widehat{c}_{2}\text{.}
\end{equation*}%
Taking into account the last estimate in (4.14), we obtain%
\begin{equation*}
\left\Vert u_{\varepsilon m}^{\tau }(t)\right\Vert _{L^{\infty }(\Omega
)}\leq \varepsilon \widehat{c}_{3}+\underset{0}{\overset{1}{\int }}\underset{%
\left\{ x:x\in \Omega \text{ }\left\vert w_{mt}^{\tau }(t,x)\right\vert
>M\right\} }{\int }f_{1}(w_{mt}^{\tau }(t,x))w_{mt}^{\tau }(t,x)dxdt
\end{equation*}%
\begin{equation}
=\varepsilon \widehat{c}_{3}+\underset{\tau }{\overset{\tau +1}{\int }}%
\underset{\left\{ x:x\in \Omega \text{ }\left\vert w_{mt}(s,x)\right\vert
>M\right\} }{\int }f_{1}(w_{mt}(s,x))w_{mt}(s,x)dxdt,\text{ \ }\forall t\in
\lbrack 0,1]\text{ and }\forall \tau \geq 0\text{.}  \tag{4.15}
\end{equation}%
Thus, by (4.12)-(4.15), we conclude that for any $\delta >0$ there exist $%
\widetilde{M}_{\delta }>0$ and $\widetilde{T}_{\delta }>0$ such that%
\begin{equation*}
w_{m}(\tau +1)-e^{-1}w_{m}(\tau )\in O_{\delta }^{\infty }(B(0,r(\widetilde{M%
}_{\delta },s)),\text{ }\forall \tau \geq \widetilde{T}_{\delta }\text{,}
\end{equation*}%
where $r(M_{\delta },s)=\frac{\widehat{c}_{1}}{2-s}(1+\left\Vert
f_{1}\right\Vert _{C[-M,M]})$, $B(0,r)=\left\{ u:u\in H^{s}(\Omega
),\left\Vert u\right\Vert _{H^{s}(\Omega )}<r\right\} $, $s\in (1,2)$, and $%
O_{\delta }^{\infty }(B(0,r)$ is $\delta $-neighbourhood of $B(0,r)$ in $%
L^{\infty }(\Omega )$. Since%
\begin{equation*}
w_{m}(\tau +n)-e^{-n}w_{m}(\tau )
\end{equation*}%
\begin{equation*}
=w_{m}(\tau +n)-e^{-1}w_{m}(\tau +n-1)+e^{-1}(w_{m}(\tau
+n-1)-e^{-1}w_{m}(\tau +n-2))+...
\end{equation*}%
\begin{equation*}
+e^{-n+1}(w_{m}(\tau +1)-e^{-1}w_{m}(\tau ))\text{,}
\end{equation*}%
by the last conclusion, we have%
\begin{equation*}
w_{m}(\tau +n)-e^{-n}w_{m}(\tau )\in O_{n\delta }^{\infty }(B(0,nr(%
\widetilde{M}_{\delta },s)),\text{ }\forall \tau \geq \widetilde{T}_{\delta }%
\text{.}
\end{equation*}%
By (4.1), for any $\varepsilon >0$ there exists $n_{\varepsilon }\in
\mathbb{N}
$ such that
\begin{equation*}
e^{-n_{\varepsilon }}\left\Vert w_{m}(\tau )\right\Vert _{L^{\infty }(\Omega
)}<\frac{\varepsilon }{3},\text{ \ }\forall \tau \geq 0\text{ and }\forall
m\in
\mathbb{N}
\text{.}
\end{equation*}%
Choosing $\delta =\frac{\varepsilon }{3n_{\varepsilon }}$ , by the last two
relations, we get%
\begin{equation}
w_{m}(\tau +n_{\varepsilon })\in O_{\frac{2\varepsilon }{3}}^{\infty
}(B(0,r_{\varepsilon }(s)),\text{ }\forall \tau \geq \widetilde{T}_{\delta }%
\text{ and }\forall s\in (1,2)\text{.}  \tag{4.16}
\end{equation}%
where $r_{\varepsilon }(s)=$ $n_{\varepsilon }r(\widetilde{M}_{\delta },s)$.
By the compact embedding $H^{s}(\Omega )\subset L^{\infty }(\Omega )$, for $%
s>1$, it follows that $B(0,r_{\varepsilon }(s))$ is relatively compact
subset of $L^{\infty }(\Omega )$. Hence, by (4.16), the set $\left\{
w_{m}(\tau +n_{\varepsilon })\right\} _{\tau \geq \widetilde{T}_{\delta }}$
and particularly, $\left\{ w_{m}(t_{m})\right\} _{m=1}^{\infty }$ has a
finite $\varepsilon $-net in $L^{\infty }(\Omega )$. Since $\varepsilon $ is
arbitrary positive number, we obtain relatively compactness of $\left\{
w_{m}(t_{m})\right\} _{m=1}^{\infty }$ in $L^{\infty }(\Omega )$, which,
together with the compactness proved in \textit{Step1}, completes the proof.
\end{proof}

Since, by (2.2) and (3.14), the problem (2.1) admits a strict Lyapunov
function $L(w(t))=\frac{1}{2}\left\Vert w_{t}(t)\right\Vert _{L^{2}(\Omega
)}^{2}+\frac{1}{2}\left\Vert \nabla w(t)\right\Vert _{L^{2}(\Omega
)}^{2}+\left\langle G(w(t)),1\right\rangle -\left\langle h,w(t)\right\rangle
$, by Lemma 4.1-4.2 and \cite[Corollary 2.29]{25}, we have Theorem 2.2.

\appendix

\section{\protect\bigskip}

\begin{lemma}
Let $\Omega \subset R^{N}$ be a bounded domain with smooth boundary. Then
for $u_{0}\in H_{0}^{1}(\Omega )\cap L^{\infty }(\Omega )$, there exists a
sequence $\left\{ u_{n}\right\} \subset C_{0}^{\infty }(\Omega )$ such that%
\begin{equation*}
\underset{n\rightarrow \infty }{\lim }\left\Vert u_{0}-u_{n}\right\Vert
_{H_{0}^{1}(\Omega )}=0\text{ \ and } \ \underset{n}{\sup }\left\Vert
u_{n}\right\Vert _{C(\overline{\Omega })}\leq \left\Vert u_{0}\right\Vert
_{L^{\infty }(\Omega )}\text{.}
\end{equation*}
\end{lemma}

\begin{proof}
By the definition of $H_{0}^{1}(\Omega )$, there exists $\left\{
v_{n}\right\} \subset C_{0}^{\infty }(\Omega )$ such that
\begin{equation*}
v_{n}\rightarrow u_{0}\text{ strongly in }H_{0}^{1}(\Omega )\text{.}
\end{equation*}%
Then there exists a subsequence\ $\left\{ v_{n_{k}}\right\} \subset
C_{0}^{\infty }(\Omega )$ such that
\begin{equation}
v_{n_{k}}\rightarrow u_{0}\text{ \ a.e. in }\Omega \text{.}  \tag{A.1}
\end{equation}%
Let $M_{0}=\left\Vert u_{0}\right\Vert _{L^{\infty }(\Omega )}$. Setting $%
w_{k}(x):=\left\{
\begin{array}{c}
-M_{0},\text{ \ \ }v_{n_{k}}(x)<-M_{0}, \\
v_{n_{k}},\text{ \ \ }\left\vert v_{n_{k}}(x)\right\vert \leq M_{0}, \\
M_{0},\text{ \ \ }v_{n_{k}}(x)>M_{0}\text{\ \ \ }%
\end{array}%
\right. $we have $w_{k}\in H_{0}^{1}(\Omega )\cap C(\overline{\Omega })$, $%
\underset{k}{\sup }\left\Vert w_{k}\right\Vert _{C(\overline{\Omega })}\leq
M_{0}$, $\overline{\left\{ x:x\in \Omega ,\text{ }w_{k}\neq 0\right\} }%
\subset \Omega $ and%
\begin{equation*}
\underset{k\rightarrow \infty }{\lim \sup }\left\Vert \nabla
(u_{0}-w_{k})\right\Vert _{L^{2}(\Omega )}^{2}=\underset{k\rightarrow \infty
}{\lim \sup }\underset{\left\{ x:\text{ }\left\vert v_{n_{k}}(x)\right\vert
>M_{0}\right\} }{\int }\left\vert \nabla u_{0}(x)\right\vert ^{2}dx
\end{equation*}%
\begin{equation}
\leq \underset{m\rightarrow \infty }{\lim \sup }\underset{\underset{k\geq m}{%
\cup }\left\{ x:\text{ }\left\vert v_{n_{k}}(x)\right\vert >M_{0}\right\} }{%
\int }\left\vert \nabla u_{0}(x)\right\vert ^{2}dx=\underset{\underset{m}{%
\cap }\underset{k\geq m}{\cup }\left\{ x:\text{ }\left\vert
v_{n_{k}}(x)\right\vert >M_{0}\right\} }{\int }\left\vert \nabla
u_{0}(x)\right\vert ^{2}dx.  \tag{A.2}
\end{equation}%
Let $A$ be a subset of $\Omega $, where a sequence $\left\{
v_{n_{k}}\right\} $ does not converge to $u_{0}$. Then we have
\begin{equation}
\underset{m}{\cap }\underset{k\geq m}{\cup }\left\{ x:\text{ }\left\vert
v_{n_{k}}(x)\right\vert >M_{0}\right\} \subset \left\{ x:\text{ }\left\vert
u_{0}(x)\right\vert \geq M_{0}\right\} \cup A.  \tag{A.3}
\end{equation}%
Since, by (A.1), $mes(A)=0$, taking into account (A.3) in (A.2), we find that%
\begin{equation*}
\underset{k\rightarrow \infty }{\lim \sup }\left\Vert \nabla
(u_{0}-w_{k})\right\Vert _{L^{2}(\Omega )}^{2}\leq \underset{\left\{ x:\text{
}\left\vert u_{0}(x)\right\vert \geq M_{0}\right\} \cup A}{\int }\left\vert
\nabla u_{0}(x)\right\vert ^{2}dx
\end{equation*}%
\begin{equation*}
=\underset{\left\{ x:\text{ }\left\vert u_{0}(x)\right\vert \geq
M_{0}\right\} }{\int }\left\vert \nabla u_{0}(x)\right\vert ^{2}dx.
\end{equation*}%
Also, considering
\begin{equation*}
mes\left( \left\{ x:\left\vert u_{0}(x)\right\vert >M_{0}\right\} \right) =0%
\text{ }\ \text{and}\underset{\left\{ x:\text{ }\left\vert
u_{0}(x)\right\vert =M_{0}\right\} }{\int }\left\vert \nabla
u_{0}(x)\right\vert ^{2}dx=0,
\end{equation*}%
in the last inequality, we obtain%
\begin{equation*}
w_{k}\rightarrow u_{0}\text{ strongly in }H_{0}^{1}(\Omega ).
\end{equation*}%
Now, to complete the proof, it is sufficient to show that for any $w\in
H_{0}^{1}(\Omega )\cap C(\overline{\Omega })$ with \newline
$\overline{\left\{ x:x\in \Omega ,\text{ }w\neq 0\right\} }\subset \Omega $,
there exists $\left\{ u_{n}\right\} \subset C_{0}^{\infty }(\Omega )$ such
that%
\begin{equation}
\underset{n\rightarrow \infty }{\lim }\left\Vert w-u_{n}\right\Vert
_{H_{0}^{1}(\Omega )}=0\text{ and }\underset{n}{\sup }\left\Vert
u_{n}\right\Vert _{C(\overline{\Omega })}\leq \left\Vert w\right\Vert _{C(%
\overline{\Omega })}  \tag{A.4}
\end{equation}%
Denoting $\overline{w}(x)=\left\{
\begin{array}{c}
w(x),\text{ }x\in \Omega \\
0,\text{ }x\in R^{N}\backslash \Omega%
\end{array}%
\right. $ and $u_{n}(x)=(\rho _{n}\ast \overline{w})(x)$, we have $\overline{%
\left\{ x:x\in \Omega ,\text{ }u_{n}\neq 0\right\} }\subset \Omega $ for
sufficiently large $n$, $\left\{ u_{n}\right\} \subset C^{\infty }(R^{N})$,
\begin{equation*}
\underset{n\rightarrow \infty }{\lim }\left\Vert \overline{w}%
-u_{n}\right\Vert _{H^{1}(R^{2})}=0\text{ and }\underset{n}{\sup }\left\Vert
u_{n}\right\Vert _{L^{\infty }(R^{2})}\leq \left\Vert w\right\Vert _{C(%
\overline{\Omega })},
\end{equation*}%
where\ $\ast $ denotes the convolution, $\rho _{n}(x)=\left\{
\begin{array}{c}
Kn^{N}e^{-\frac{1}{1-n^{2}\left\vert x\right\vert ^{2}}},\text{ \ }%
\left\vert x\right\vert <\frac{1}{n}, \\
0,\text{ \ \ \ \ \ \ \ \ \ \ \ \ \ \ \ }\left\vert x\right\vert \geq \frac{1%
}{n}%
\end{array}%
\right. $, $n\in
\mathbb{N}
$ and $K^{-1}=\underset{\left\{ x:\text{ }\left\vert x\right\vert <1\right\}
}{\int }e^{-\frac{1}{1-\left\vert x\right\vert ^{2}}}dx$. So, the
restriction of the sequence $\left\{ u_{n}\right\} $ to $\Omega $ satisfies
(A.4), for sufficiently large $n$.
\end{proof}

\begin{lemma}
Let $Q\subset R^{N}$ be a measurable set with finite measure and $\varphi $
a continuous function on $R$ such that $\varphi (s)s\geq 0$ for every $s\in
R $. If
\begin{equation*}
u_{n}\rightarrow u\text{ a.e. in }Q\text{ \ and }\underset{n}{\sup }\underset%
{Q}{\int }\varphi (u_{n}(x))u_{n}(x)dx<\infty \text{,}
\end{equation*}%
then
\begin{equation}
\underset{n\rightarrow \infty }{\lim }\left\Vert \varphi (u_{n})-\varphi
(u)\right\Vert _{L^{1}(Q)}=0.  \tag{A.5}
\end{equation}
\end{lemma}

\begin{proof}
By the continuity of $\varphi $, we have
\begin{equation*}
\underset{Q}{\int }\left\vert \varphi (u_{n}(x))\right\vert dx\leq \underset{%
\left\{ x:x\in Q,\text{ }\left\vert u_{n}(x)\right\vert \leq 1\right\} }{%
\int }\left\vert \varphi (u_{n}(x))\right\vert dx
\end{equation*}%
\begin{equation*}
+\underset{\left\{ x:x\in Q,\text{ }\left\vert u_{n}(x)\right\vert
>1\right\} }{\int }\left\vert \varphi (u_{n}(x))\right\vert dx\leq
\left\Vert \varphi \right\Vert _{C[-1,1]}mes(Q)
\end{equation*}%
\begin{equation*}
+\underset{Q}{\int }\varphi (u_{n}(x))u_{n}(x)dx.
\end{equation*}%
So, $\varphi (u_{n})\in L^{1}(Q)$. Applying Fatou's lemma, we have
\begin{equation*}
\underset{Q}{\int }\varphi (u(x))u(x)dx\leq \underset{n\rightarrow \infty }{%
\lim \inf }\underset{Q}{\int }\varphi (u_{n}(x))u_{n}(x)dx<\infty ,
\end{equation*}%
which, as shown above, yields $\varphi (u)\in L^{1}(Q).$ Also, by Egorov's
theorem, for any $\varepsilon >0$, there exists $Q_{\varepsilon }\subset Q$
such that $mes(Q\backslash Q_{\varepsilon })<\varepsilon $ and%
\begin{equation*}
u_{n}\rightarrow u\text{ uniformly in }Q_{\varepsilon }.
\end{equation*}%
Now, denote $A_{k}=\left\{ x:x\in Q,\left\vert u(x)\right\vert \leq
k\right\} $ and $A_{nk}=\left\{ x:x\in Q,\left\vert u_{n}(x)\right\vert \leq
k\right\} $, for $k>1$. By the last approximation, we get
\begin{equation}
\varphi (u_{n})\rightarrow \varphi (u)\text{ uniformly in }Q_{\varepsilon
}\cap A_{k}.  \tag{A.6}
\end{equation}%
Since, for sufficiently large $n$,
\begin{equation*}
\underset{Q_{\varepsilon }\backslash A_{k}}{\int }\left\vert \varphi
(u_{n}(x))-\varphi (u(x))\right\vert dx\leq \frac{1}{k-1}\underset{Q}{\int }%
\varphi (u_{n}(x))u_{n}(x)dx
\end{equation*}%
\begin{equation*}
+\frac{1}{k}\underset{Q}{\int }\varphi (u(x))u(x)dx,
\end{equation*}%
using (A.6), we obtain%
\begin{equation*}
\underset{n\rightarrow \infty }{\lim \sup }\underset{Q}{\int }\left\vert
\varphi (u_{n}(x))-\varphi (u(x))\right\vert dx\leq \underset{n\rightarrow
\infty }{\lim \sup }\underset{Q_{\varepsilon }\cap A_{k}}{\int }\left\vert
\varphi (u_{n}(x))-\varphi (u(x))\right\vert dx
\end{equation*}%
\begin{equation*}
+\underset{n\rightarrow \infty }{\lim \sup }\underset{Q_{\varepsilon
}\backslash A_{k}}{\int }\left\vert \varphi (u_{n}(x))-\varphi
(u(x))\right\vert dx+\underset{n\rightarrow \infty }{\lim \sup }\underset{%
Q\backslash Q_{\varepsilon }}{\int }\left\vert \varphi (u_{n}(x))-\varphi
(u(x))\right\vert dx
\end{equation*}%
\begin{equation*}
\leq \frac{1}{k-1}\underset{n\rightarrow \infty }{\lim \sup }\underset{Q}{%
\int }\varphi (u_{n}(x))u_{n}(x)dx+\frac{1}{k}\underset{Q}{\int }\varphi
(u(x))u(x)dx
\end{equation*}%
\begin{equation*}
+\underset{n\rightarrow \infty }{\lim \sup }\underset{Q\backslash
Q_{\varepsilon }}{\int }\left\vert \varphi (u_{n}(x))-\varphi
(u(x))\right\vert dx.
\end{equation*}%
Passing to the limit as $k\rightarrow \infty $ and $\varepsilon \rightarrow
0 $, we find
\begin{equation}
\underset{n\rightarrow \infty }{\lim \sup }\underset{Q}{\int }\left\vert
\varphi (u_{n}(x))-\varphi (u(x))\right\vert dx\leq \underset{\varepsilon
\rightarrow 0}{\lim \sup }\underset{n\rightarrow \infty }{\lim \sup }%
\underset{Q\backslash Q_{\varepsilon }}{\int }\left\vert \varphi
(u_{n}(x))-\varphi (u(x))\right\vert dx.  \tag{A.7}
\end{equation}%
Now, let us estimate the right hand side of (A.7).%
\begin{equation*}
\underset{Q\backslash Q_{\varepsilon }}{\int }\left\vert \varphi
(u_{n}(x))-\varphi (u(x))\right\vert dx\leq \underset{Q\backslash
Q_{\varepsilon }}{\int }\left\vert \varphi (u_{n}(x))\right\vert dx+\underset%
{Q\backslash Q_{\varepsilon }}{\int }\left\vert \varphi (u(x))\right\vert dx
\end{equation*}%
\begin{equation*}
\leq \frac{1}{k}\underset{Q}{\int }\varphi (u_{n}(x))u_{n}(x)dx+\frac{1}{k}%
\underset{Q}{\int }\varphi (u(x))u(x)dx
\end{equation*}%
\begin{equation*}
+\underset{\left( Q\backslash Q_{\varepsilon }\right) \cap A_{nk}}{\int }%
\left\vert \varphi (u_{n}(x))\right\vert dx+\underset{\left( Q\backslash
Q_{\varepsilon }\right) \cap A_{k}}{\int }\left\vert \varphi
(u(x))\right\vert dx
\end{equation*}%
\begin{equation*}
\leq \frac{1}{k}\underset{Q}{\int }\varphi (u_{n}(x))u_{n}(x)dx+\frac{1}{k}%
\underset{Q}{\int }\varphi (u(x))u(x)dx
\end{equation*}%
\begin{equation*}
+2\left\Vert \varphi \right\Vert _{C[-k,k]}mes(Q\backslash Q_{\varepsilon })
\end{equation*}%
and consequently%
\begin{equation*}
\underset{\varepsilon \rightarrow 0}{\lim \sup }\underset{n\rightarrow
\infty }{\lim \sup }\underset{Q\backslash Q_{\varepsilon }}{\int }\left\vert
\varphi (u_{n}(x))-\varphi (u(x))\right\vert dx\leq \frac{1}{k}\underset{Q}{%
\int }\varphi (u(x))u(x)dx
\end{equation*}%
\begin{equation*}
+\frac{1}{k}\underset{n\rightarrow \infty }{\lim \sup }\underset{Q}{\int }%
\varphi (u_{n}(x))u_{n}(x)dx.
\end{equation*}%
Passing to the limit as $k\rightarrow \infty $, we obtain%
\begin{equation*}
\underset{\varepsilon \rightarrow 0}{\lim \sup }\underset{n\rightarrow
\infty }{\lim \sup }\underset{Q\backslash Q_{\varepsilon }}{\int }\left\vert
\varphi (u_{n}(x))-\varphi (u(x))\right\vert dx=0,
\end{equation*}%
which, together with (A.7), gives us (A.5).
\end{proof}

\begin{lemma}
Let $\Omega \subset R^{N}$ be a bounded domain with smooth boundary and \ $%
h\in L^{1}((0,T)\times \Omega )$. If $u\in C_{s}(0,T;L^{1}(\Omega ))\cap
L^{2}(0,T;H_{0}^{1}(\Omega ))$ with $\underset{t\searrow 0}{\lim }\left\Vert
u(t)\right\Vert _{L^{1}(\Omega )}$ $=0$, is a solution of%
\begin{equation}
\left\{
\begin{array}{c}
u_{t}(t,x)-\Delta u(s,x)=h(t,x),\text{ \ \ }(t,x)\in (0,T)\times \Omega ,%
\text{ \ \ \ \ \ \ \ \ \ \ \ \ \ } \\
u(t,x)=0,\text{ \ \ }(t,x)\in ((0,T)\times \partial \Omega )\cup (\left\{
0\right\} \times \Omega )\text{\ \ \ \ \ \ \ \ \ \ \ \ \ \ }%
\end{array}%
\right.  \tag{A.8}
\end{equation}%
then
\begin{equation}
\left\vert u(t,x)\right\vert \leq \frac{1}{(4\pi )^{\frac{N}{2}}}\underset{0}%
{\overset{t}{\int }}\frac{1}{(t-s)^{\frac{N}{2}}}\underset{\Omega }{\int }%
e^{-\frac{\left\vert x-y\right\vert ^{2}}{4(t-s)}}\left\vert
h(s,y)\right\vert dyds,\text{ \ a.e. in }(0,T)\times \Omega \text{,}
\tag{A.9}
\end{equation}%
where $C_{s}(0,T;L^{1}(\Omega ))=\left\{ u:u\in L^{\infty }(0,T;L^{1}(\Omega
)),\int\limits_{\Omega }\varphi (x)u(\cdot ,x)dx\in C[0,T]\text{ for every}%
\right. $\newline
$\left. \varphi \in L^{\infty }(\Omega )\right\} $.
\end{lemma}

\begin{proof}
Let $\theta _{tn}(\tau )=\left\{
\begin{array}{c}
0,\text{ \ \ \ \ \ \ \ \ \ \ \ \ \ \ \ \ \ \ \ \ }\forall \tau \in \lbrack 0,%
\frac{1}{2n}), \\
2n(\tau -\frac{1}{2n}),\text{ \ \ \ \ \ \ \ }\forall \tau \in \lbrack \frac{1%
}{2n},\frac{1}{n}), \\
1,\text{ \ \ \ \ \ \ \ \ \ \ \ \ \ \ \ \ \ \ }\forall \tau \in \lbrack \frac{%
1}{n},t-\frac{1}{n}), \\
2n(t-\frac{1}{2n}-\tau ),\text{ \ \ \ \ \ \ \ \ \ }\forall \tau \in \lbrack
t-\frac{1}{n},t-\frac{1}{2n}],\text{\ \ } \\
0,\text{ \ \ \ \ \ \ \ \ \ \ \ \ \ \ \ \ \ \ }\forall \tau \in (t-\frac{1}{2n%
},t]%
\end{array}%
\right. $. Also denote\newline
$\varphi _{m}(x)=\left\{
\begin{array}{c}
-1,\text{ \ \ }x<-\frac{1}{m}, \\
mx,\text{ \ \ }\left\vert x\right\vert \leq \frac{1}{m}, \\
1,\text{ \ \ }x>\frac{1}{m}%
\end{array}%
\right. $ , $\delta _{k}(t)=\left\{
\begin{array}{c}
Mk^{2}e^{-\frac{1}{1-k^{2}\left\vert t\right\vert ^{2}}},\text{ \ }%
\left\vert t\right\vert <\frac{1}{k}, \\
0,\text{ \ \ \ \ \ \ \ \ \ \ \ \ \ \ \ }\left\vert t\right\vert \geq \frac{1%
}{k}%
\end{array}%
\right. ,$ where $m,k\in
\mathbb{N}
$ and $M^{-1}=\underset{-1}{\overset{1}{\int }}e^{-\frac{1}{1-t^{2}}}dt$.
Testing (A.8)$_{1}$ by $(\theta _{\tau n}(t)\delta _{k}(s-t)$ on $(0,\tau )$%
, we get%
\begin{equation*}
\frac{\partial }{\partial s}(\theta _{\tau n}u(.,x)\ast \delta
_{k})(s)-(u(.,x)\theta _{\tau n}^{\prime }\ast \delta _{k})(s)-\Delta
(\theta _{\tau n}u(.,x)\ast \delta _{k})(s)
\end{equation*}%
\begin{equation*}
=(h(.,x)\theta _{\tau n}\ast \delta _{k})(s),\text{ \ }(s,x)\in (0,\tau
)\times \Omega ,\text{ }\tau \in (\frac{1}{2n},T).
\end{equation*}%
Testing the last equation by $\varphi _{m}((\theta _{\tau n}u)\ast \delta
_{k})(s,x))$ on $(\frac{1}{3n},\tau )\times \Omega $, we obtain%
\begin{equation*}
\underset{\Omega }{\int }\Phi _{m}((\theta _{\tau n}u(.,x))\ast \delta
_{k})(\tau ))dx-\underset{\Omega }{\int }\Phi _{m}((\theta _{\tau
n}u(.,x))\ast \delta _{k})(\frac{1}{3n}))dx
\end{equation*}%
\begin{equation*}
-\underset{\frac{1}{3n}}{\overset{t}{\int }}\underset{\Omega }{\int }%
(u(.,x)\theta _{\tau n}^{\prime }\ast \delta _{k})(s)\varphi _{m}((\theta
_{\tau n}u(.,x))\ast \delta _{k})(s))dxds
\end{equation*}%
\begin{equation}
\leq \underset{\frac{1}{3n}}{\overset{t}{\int }}\underset{\Omega }{\int }%
\left\vert (h(.,x)\theta _{\tau n}\ast \delta _{k})(s)\varphi _{m}((\theta
_{\tau n}u(.,x))\ast \delta _{k})(s))\right\vert dxds,  \tag{A.10}
\end{equation}%
where $\Phi _{m}(y)=\underset{0}{\overset{y}{\int }}\varphi _{m}(x)dx.$ By
the definition of $\theta _{tn}$, $\delta _{k}$ and $\varphi _{m}$, we have%
\begin{equation*}
\left\{
\begin{array}{c}
\underset{k\rightarrow \infty }{\lim }\underset{\Omega }{\int }\Phi
_{m}((\theta _{\tau n}u(.,x)\ast \delta _{k})(\frac{1}{3n}))dx=0, \\
\underset{k\rightarrow \infty }{\lim }\underset{\Omega }{\int }\Phi
_{m}((\theta _{\tau n}u(.,x)\ast \delta _{k})(\tau ))dx=0, \\
\underset{k\rightarrow \infty }{\lim }\underset{\frac{1}{3n}}{\overset{\tau }%
{\int }}\underset{\Omega }{\int }(u(.,x)\theta _{\tau n}^{\prime }\ast
\delta _{k})(s)\varphi _{m}((\theta _{\tau n}u(.,x))\ast \delta _{k})(s))dxds
\\
=\underset{\frac{1}{2n}}{\overset{\tau }{\int }}\underset{\Omega }{\int }%
\theta _{\tau n}^{\prime }(s)u(s,x)\varphi _{m}(\theta _{\tau
n}(s)u(s,x))dxds, \\
\underset{k\rightarrow \infty }{\lim \sup }\underset{\frac{1}{3n}}{\overset{%
\tau }{\int }}\underset{\Omega }{\int }\left\vert (h(.,x)\theta _{\tau
n}\ast \delta _{k})(s)\varphi _{m}((\theta _{\tau n}u(.,x))\ast \delta
_{k})(s))\right\vert dxds \\
\leq \underset{k\rightarrow \infty }{\lim \sup }\underset{\frac{1}{3n}}{%
\overset{\tau }{\int }}\underset{\Omega }{\int }\left\vert (h(.,x)\theta
_{\tau n}\ast \delta _{k})(s))\right\vert dxds=\underset{\frac{1}{2n}}{%
\overset{\tau }{\int }}\underset{\Omega }{\int }\left\vert h(s,x)\theta
_{\tau n}(s)\right\vert dxds.%
\end{array}%
\right.
\end{equation*}%
Thus, passing to the limit in (A.10) as $k\rightarrow \infty $, we obtain%
\begin{equation*}
-\underset{\frac{1}{2n}}{\overset{\tau }{\int }}\underset{\Omega }{\int }%
u(s,x)\theta _{\tau n}^{\prime }(s)\varphi _{m}(\theta _{\tau
n}(s)u(s,x))dxds\leq \underset{\frac{1}{2n}}{\overset{\tau }{\int }}\underset%
{\Omega }{\int }\left\vert h(s,x)\theta _{\tau n}(s)\right\vert dxds
\end{equation*}%
Now, take the limit in the last inequality as $m\rightarrow \infty $, we find%
\begin{equation}
2n\underset{\tau -\frac{1}{n}}{\overset{\tau -\frac{1}{2n}}{\int }}\underset{%
\Omega }{\int }\left\Vert u(s)\right\Vert _{L^{1}(\Omega )}ds-2n\underset{%
\frac{1}{2n}}{\overset{\frac{1}{n}}{\int }}\left\Vert u(s)\right\Vert
_{L^{1}(\Omega )}ds\leq \underset{0}{\overset{\tau }{\int }}\left\Vert
h(s)\right\Vert _{L^{1}(\Omega )}ds\text{.}  \tag{A.11}
\end{equation}%
Since $u\in C_{s}(0,T;L_{1}(\Omega )),$ by the weak lower semi-continuity of
the norm, it follows that
\begin{equation*}
\underset{n\rightarrow \infty }{\lim \inf }2n\underset{t-\frac{1}{n}}{%
\overset{t-\frac{1}{2n}}{\int }}\underset{\Omega }{\int }\left\Vert u(\tau
)\right\Vert _{L^{1}(\Omega )}d\tau \geq \left\Vert u(t)\right\Vert
_{L^{1}(\Omega )}\text{.}
\end{equation*}%
Also, by the condition $\underset{t\searrow 0}{\lim }\left\Vert
u(t)\right\Vert _{L^{1}(\Omega )}=0$, we have%
\begin{equation*}
\underset{n\rightarrow \infty }{\lim }2n\underset{\frac{1}{2n}}{\overset{%
\frac{1}{n}}{\int }}\left\Vert u(\tau )\right\Vert _{L^{1}(\Omega )}d\tau =0%
\text{.}
\end{equation*}%
Hence, passing to the limit in (A.11) as $n\rightarrow \infty $, we get%
\begin{equation}
\left\Vert u(t)\right\Vert _{L^{1}(\Omega )}\leq \left\Vert h\right\Vert
_{L^{1}((0,t)\times \Omega )},\text{ \ }\forall t\in \lbrack 0,T]\text{.}
\tag{A.12}
\end{equation}%
Define $H(t,x)=\left\{
\begin{array}{c}
h(t,x),\text{ \ }(t,x)\in (0,T)\times \Omega \text{ \ } \\
0,\text{ \ }(t,x)\in R^{N+1}\backslash (0,T)\times \Omega%
\end{array}%
\right. $, $H_{m}(t,x)=(H\ast \rho _{m})(t,x)$ and $\overline{H}%
_{m}(t,x)=(\left\vert H\right\vert \ast \rho _{m})(t,x)$. Since $H_{m}\in
C_{0}(R^{N+1})$ and $\overline{H}_{m}\in C_{0}(R^{N+1})$, the problems%
\begin{equation*}
\left\{
\begin{array}{c}
u_{mt}(t,x)-\Delta u_{m}(s,x)=H_{m}(t,x),\text{ \ \ }(t,x)\in (0,T)\times
\Omega ,\text{ \ \ } \\
u_{m}(t,x)=0,\text{ \ \ \ \ \ \ \ \ \ \ \ }(t,x)\in ((0,T)\times \partial
\Omega )\cup (\left\{ 0\right\} \times \Omega )\text{\ \ \ }%
\end{array}%
\right.
\end{equation*}%
and
\begin{equation*}
\left\{
\begin{array}{c}
v_{mt}(t,x)-\Delta v_{m}(s,x)=\overline{H}_{m}(t,x),\text{ \ \ }(t,x)\in
(0,T)\times R^{N},\text{ \ } \\
v_{m}(t,x)=0,\text{ \ \ \ \ \ \ \ \ \ \ \ \ \ \ \ \ \ \ \ \ \ \ \ \ \ \ \ \ }%
(t,x)\in \left\{ 0\right\} \times R^{N}\text{\ \ \ \ \ \ \ }%
\end{array}%
\right.
\end{equation*}%
have unique smooth classical solutions. By (A.12), it follows that%
\begin{equation}
\left\Vert u(t)-u_{m}(t)\right\Vert _{L^{1}(\Omega )}\leq \left\Vert
H-H_{m}\right\Vert _{L^{1}((0,t)\times \Omega )},\text{ \ }\forall t\in
\lbrack 0,T]\text{.}  \tag{A.13}
\end{equation}%
On the other hand, applying Duhamel's principle (see \cite[p. 49]{26}), we
get
\begin{equation}
v_{m}(t,x)=\frac{1}{\left( 4\pi \right) ^{\frac{N}{2}}}\underset{0}{\overset{%
t}{\int }}\frac{1}{\left( t-s\right) ^{\frac{N}{2}}}\underset{R^{N}}{\int }%
e^{-\frac{\left\vert x-y\right\vert ^{2}}{4(t-s)}}\overline{H}_{m}(s,y)dyds.
\tag{A.14}
\end{equation}%
Denoting $w_{m}(t,x)=v_{m}(t,x)-u_{m}(t,x)$, we have%
\begin{equation*}
\left\{
\begin{array}{c}
w_{mt}(t,x)-\Delta w_{m}(s,x)\geq 0,\text{ \ \ }(t,x)\in (0,T)\times \Omega ,%
\text{ \ } \\
w_{m}(t,x)\geq 0,\text{ \ \ \ \ \ \ \ \ \ \ \ \ \ \ \ \ \ }(t,x)\in
(0,T)\times \partial \Omega ,\text{\ \ \ } \\
w_{m}(t,x)=0,\text{ \ \ \ \ \ \ \ \ \ \ \ \ \ \ \ \ \ \ \ \ }(t,x)\in
\left\{ 0\right\} \times \Omega \text{.\ \ \ \ \ }%
\end{array}%
\right.
\end{equation*}%
So, by the maximum principle, it follows that
\begin{equation*}
w_{m}(t,x)\geq 0\text{ or }v_{m}(t,x)\geq u_{m}(t,x)\text{, \ }\forall
(t,x)\in \lbrack 0,T]\times \overline{\Omega }\text{.}
\end{equation*}%
By the similar way, one can show that%
\begin{equation*}
u_{m}(t,x)\geq -v_{m}(t,x)\text{, \ }\forall (t,x)\in \lbrack 0,T]\times
\overline{\Omega }\text{.}
\end{equation*}%
Hence,
\begin{equation*}
\left\vert u_{m}(t,x)\right\vert \leq v_{m}(t,x)\text{, \ }\forall (t,x)\in
\lbrack 0,T]\times \overline{\Omega }\text{,}
\end{equation*}%
which, together with (A.14), yields%
\begin{equation}
\left\vert u_{m}(t,x)\right\vert \leq \frac{1}{\left( 4\pi \right) ^{\frac{N%
}{2}}}\underset{0}{\overset{t}{\int }}\frac{1}{\left( t-s\right) ^{\frac{N}{2%
}}}\underset{R^{N}}{\int }e^{-\frac{\left\vert x-y\right\vert ^{2}}{4(t-s)}}%
\overline{H}_{m}(s,y)dyds\text{, \ }\forall (t,x)\in \lbrack 0,T]\times
\overline{\Omega }\text{.}  \tag{A.15}
\end{equation}%
Since%
\begin{equation*}
\left\Vert \underset{0}{\overset{t}{\int }}\frac{1}{\left( t-s\right) ^{%
\frac{N}{2}}}\underset{R^{N}}{\int }e^{-\frac{\left\vert x-y\right\vert ^{2}%
}{4(t-s)}}\overline{H}_{m}(s,y)dyds-\underset{0}{\overset{t}{\int }}\frac{1}{%
\left( t-s\right) ^{\frac{N}{2}}}\underset{\Omega }{\int }e^{-\frac{%
\left\vert x-y\right\vert ^{2}}{4(t-s)}}\left\vert h(s,y)\right\vert
dyds\right\Vert _{L^{1}(\Omega )}
\end{equation*}%
\begin{equation*}
\leq \underset{R^{N}}{\int }\underset{0}{\overset{t}{\int }}\frac{1}{\left(
t-s\right) ^{\frac{N}{2}}}\underset{R^{N}}{\int }e^{-\frac{\left\vert
x-y\right\vert ^{2}}{4(t-s)}}\left\vert \overline{H}_{m}(s,y)-\left\vert
H(s,y)\right\vert \right\vert dydsdx
\end{equation*}%
\begin{equation*}
=\underset{0}{\overset{t}{\int }}\underset{R^{N}}{\int }\left\vert \overline{%
H}_{m}(s,y)-\left\vert H(s,y)\right\vert \right\vert \frac{1}{\left(
t-s\right) ^{\frac{N}{2}}}\underset{R^{N}}{\int }e^{-\frac{\left\vert
x-y\right\vert ^{2}}{4(t-s)}}dxdyds
\end{equation*}%
\begin{equation*}
=\left( 4\pi \right) ^{\frac{N}{2}}\underset{0}{\overset{t}{\int }}\underset{%
R^{N}}{\int }\left\vert \overline{H}_{m}(s,y)-\left\vert H(s,y)\right\vert
\right\vert dyds
\end{equation*}%
\begin{equation}
=\left( 4\pi \right) ^{\frac{N}{2}}\left\Vert \left\vert H\right\vert
-\left\vert H\right\vert \ast \rho _{m}\right\Vert _{L^{1}((0,t)\times
\Omega )},\text{ \ \ \ }\forall t\in \lbrack 0,T],  \tag{A.16}
\end{equation}%
passing to the limit in (A.15) and taking into account (A.13) and (A.16), we
obtain (A.9).
\end{proof}

\bigskip

\begin{lemma}
Let $Q\subset R^{N}$ be a measurable set of finite measure and $\varphi $ an
increasing continuous function such that $\varphi (0)=0$. Further assume
that $v_{i}\in L^{\infty }(0,T;L^{\infty }(Q))$ and $v_{it}\in
L^{1}(0,T;L^{1}(Q))$, $i=1,2$. If
\begin{equation*}
\underset{0}{\overset{T}{\int }}\underset{Q}{\int }\left( \left\vert \varphi
(v_{it}(t,x))\right\vert +\left\vert \varphi (-v_{it}(t,x))\right\vert
\right) \left\vert v_{it}(t,x)\right\vert dxdt<\infty ,\text{ }i=1,2,
\end{equation*}%
then%
\begin{equation}
\underset{h\searrow 0}{\lim \inf }\underset{s}{\overset{t}{\int }}\underset{s%
}{\overset{\sigma }{\int }}\underset{Q}{\int }\left( \varphi (v_{2t}(\tau
,x))-\varphi (v_{1t}(\tau ,x))\right) \frac{w(\tau +h,x)-w(\tau -h,x)}{2h}%
dxd\tau d\sigma \geq 0,\text{ }  \tag{A.17}
\end{equation}%
for every $[s,t]\subset (0,T)$, where $w(t,x)=v_{2}(t,x)-v_{1}(t,x)$.
\end{lemma}

\begin{proof}
By the conditions of the lemma, it follows that
\begin{equation*}
\underset{0}{\overset{T}{\int }}\underset{Q}{\int }\left\vert \varphi
(v_{it}(t,x))\right\vert dxdt=\underset{0}{\overset{T}{\int }}\underset{%
\left\{ x:x\in Q,\left\vert v_{it}(t,x)\right\vert >1\right\} }{\int }%
\left\vert \varphi (v_{it}(t,x))\right\vert dxdt
\end{equation*}%
\begin{equation*}
+\underset{0}{\overset{T}{\int }}\underset{\left\{ x:x\in Q,\left\vert
v_{it}(t,x)\right\vert \leq 1\right\} }{\int }\left\vert \varphi
(v_{it}(t,x))\right\vert dxdt
\end{equation*}%
\begin{equation*}
\leq \underset{0}{\overset{T}{\int }}\underset{\left\{ x:x\in Q,\left\vert
v_{it}(t,x)\right\vert >1\right\} }{\int }\left\vert \varphi
(v_{it}(t,x))\right\vert \left\vert v_{it}(t,x)\right\vert dxdt
\end{equation*}%
\begin{equation*}
+Tmes(Q)\left\Vert \varphi \right\Vert _{C[-1,1]}<\infty ,\text{ \ }i=1,2.
\end{equation*}%
Therefore, the integral in (A.17) is well defined.

Now, let us denote $M_{\varphi }(t)=\underset{0}{\overset{t}{\int }}\varphi
(s)ds$, $N_{\varphi }(t)=\underset{0}{\overset{t}{\int }}\varphi ^{-1}(s)ds$
and $g(t)=-\varphi (-t)$. By Young's inequality (see \cite[p. 12]{27}), we
have%
\begin{equation}
uv\leq M_{\varphi }(u)+N_{\varphi }(v),\text{ \ \ }\forall u,v\geq 0\text{.}
\tag{A.18}
\end{equation}%
If $u<0$ and $v<0$, then again by Young's inequality,
\begin{equation*}
uv=-u(-v)\leq M_{g}(-u)+N_{g}(-v)=M_{\varphi }(u)+N_{\varphi }(u)\text{.}
\end{equation*}%
Since the right hand side of (A.18) is nonnegative for all $u,v\in R$, in
the case $uv<0$, this inequality is trivial. Hence, the inequality (A.18)
holds for all $u,v\in R$. Therefore we conclude that%
\begin{equation}
uv=-u(-v)\geq -M_{\varphi }(u)-N_{\varphi }(-v),\text{ \ }\forall u,v\in R.
\tag{A.19}
\end{equation}

Denote $\varphi _{M}(x)=\left\{
\begin{array}{c}
-M,\text{ \ \ }x<-M, \\
\varphi (x),\text{ \ \ }\left\vert x\right\vert \leq M, \\
M,\text{ \ \ }x>M%
\end{array}%
\right. $ . Since, by definition,%
\begin{equation*}
\underset{h\searrow 0}{\lim }\left\Vert \frac{w(\cdot +h,)-w(\cdot -h,)}{2h}%
-w_{t}(\cdot ,)\right\Vert _{L^{1}((0,t)\times Q)}=0,\text{ }\forall t\in
(0,T),
\end{equation*}%
we have%
\begin{equation*}
\frac{w(t+h,x)-w(t-h,x)}{2h}\rightarrow w_{t}(t,x)\text{ in the measure as }%
h\searrow 0\text{.}
\end{equation*}%
Hence, applying Lebesgue's convergence theorem and taking into account the
monotonicity of $\varphi _{M}$, we get%
\begin{equation*}
\underset{h\searrow 0}{\lim }\underset{s}{\overset{t}{\int }}\underset{s}{%
\overset{\sigma }{\int }}\underset{Q}{\int }\left( \varphi _{M}(v_{2t}(\tau
,x))-\varphi _{M}(v_{1t}(\tau ,x))\right) \frac{w(\tau +h,x)-w(\tau -h,x)}{2h%
}dxd\tau d\sigma
\end{equation*}%
\begin{equation*}
=\underset{s}{\overset{t}{\int }}\underset{s}{\overset{\sigma }{\int }}%
\underset{Q}{\int }\left( \varphi _{M}(v_{2t}(\tau ,x))-\varphi
_{M}(v_{1t}(\tau ,x))\right) w_{t}(\tau ,x)dxd\tau d\sigma \geq 0
\end{equation*}%
and consequently%
\begin{equation*}
\underset{h\searrow 0}{\lim \inf }\underset{s}{\overset{t}{\int }}\underset{s%
}{\overset{\sigma }{\int }}\underset{Q}{\int }\left( \varphi (v_{2t}(\tau
,x))-\varphi (v_{1t}(\tau ,x))\right) \frac{w(\tau +h,x)-w(\tau -h,x)}{2h}%
dxd\tau d\sigma
\end{equation*}%
\begin{equation*}
\underset{h\searrow 0}{\geq \lim \inf }\underset{s}{\overset{t}{\int }}%
\underset{s}{\overset{\sigma }{\int }}\underset{Q}{\int }\left( \varphi
(v_{2t}(\tau ,x))-\varphi _{M}(v_{2t}(\tau ,x))\right) \frac{v_{2}(\tau
+h,x)-v_{2}(\tau -h,x)}{2h}dxd\tau
\end{equation*}%
\begin{equation*}
+\underset{h\searrow 0}{\lim \inf }\underset{s}{\overset{t}{\int }}\underset{%
s}{\overset{\sigma }{\int }}\underset{Q}{\int }\left( \varphi (v_{2t}(\tau
,x))-\varphi _{M}(v_{2t}(\tau ,x))\right) \frac{v_{1}(\tau -h,x)-v_{1}(\tau
+h,x)}{2h}dxd\tau d\sigma
\end{equation*}%
\begin{equation*}
+\underset{h\searrow 0}{\lim \inf }\underset{s}{\overset{t}{\int }}\underset{%
s}{\overset{\sigma }{\int }}\underset{Q}{\int }\left( \varphi (v_{1t}(\tau
,x))-\varphi _{M}(v_{1t}(\tau ,x))\right) \frac{v_{1}(\tau +h,x)-v_{1}(\tau
-h,x)}{2h}dxd\tau d\sigma
\end{equation*}%
\begin{equation*}
+\underset{h\searrow 0}{\lim \inf }\underset{s}{\overset{t}{\int }}\underset{%
s}{\overset{\sigma }{\int }}\underset{Q}{\int }\left( \varphi (v_{1t}(\tau
,x))-\varphi _{M}(v_{1t}(\tau ,x))\right) \frac{v_{2}(\tau -h,x)-v_{2}(\tau
+h,x)}{2h}dxd\tau d\sigma
\end{equation*}%
\begin{equation}
=:I_{1}^{M}(s,t)+I_{2}^{M}(s,t)+I_{3}^{M}(s,t)+I_{4}^{M}(s,t),\text{ \ \ }%
\forall \lbrack s,t]\subset (0,T)\text{ and }\forall M>0.  \tag{A.20}
\end{equation}%
Now, let us estimate each $I_{i}^{M}(s,t)$ ($i=1,2,3,4$). By (A.19) and
Jensen's inequality for convex functions (see \cite[p. 62]{27}), we have%
\begin{equation*}
I_{1}^{M}(t)=\underset{h\searrow 0}{\lim \inf }\underset{s}{\overset{t}{\int
}}\underset{s}{\overset{\sigma }{\int }}\underset{Q}{\int }\left( \varphi
(v_{2t}(\tau ,x))-\varphi _{m}(v_{2t}(\tau ,x))\right) \frac{v_{2}(\tau
+h,x)-v_{2}(\tau -h,x)}{2h}dxd\tau d\sigma
\end{equation*}%
\begin{equation*}
\geq -\underset{s}{\overset{t}{\int }}\underset{s}{\overset{\sigma }{\int }}%
\underset{\left\{ x:x\in Q,\left\vert v_{2t}(\tau ,x)\right\vert >M\right\} }%
{\int }N_{\varphi }\left( \varphi (v_{2t}(\tau ,x))-\varphi _{m}(v_{2t}(\tau
,x))\right) dxd\tau d\sigma
\end{equation*}%
\begin{equation*}
-\underset{h\searrow 0}{\lim \sup }\underset{s}{\overset{t}{\int }}\underset{%
s}{\overset{\sigma }{\int }}\underset{\left\{ x:x\in Q,\left\vert
v_{2t}(\tau ,x)\right\vert >M\right\} }{\int }M_{\varphi }\left( \frac{1}{2}%
\underset{-1}{\overset{1}{\int }}-v_{2t}(\tau +\mu h,x)d\mu \right) dxd\tau
d\sigma
\end{equation*}%
\begin{equation*}
\geq -\underset{s}{\overset{t}{\int }}\underset{s}{\overset{\sigma }{\int }}%
\underset{\left\{ x:x\in Q,\left\vert v_{2t}(\tau ,x)\right\vert >M\right\} }%
{\int }N_{\varphi }\left( \varphi (v_{2t}(\tau ,x))\right) dxd\tau d\sigma
\end{equation*}%
\begin{equation}
-\frac{1}{2}\underset{h\searrow 0}{\lim \sup }\underset{s}{\overset{t}{\int }%
}\underset{s}{\overset{\sigma }{\int }}\underset{\left\{ x:x\in Q,\left\vert
v_{2t}(\tau ,x)\right\vert >M\right\} }{\int }\underset{-1}{\overset{1}{\int
}}M_{\varphi }(-v_{2t}(\tau +\mu h,x))d\mu dxd\tau d\sigma .  \tag{A.21}
\end{equation}%
By the definition of $N_{\varphi }$ and $M_{\varphi }$, we obtain%
\begin{equation*}
\underset{s}{\overset{t}{\int }}\underset{s}{\overset{\sigma }{\int }}%
\underset{\left\{ x:x\in Q,\left\vert v_{2t}(\tau ,x)\right\vert >M\right\} }%
{\int }N_{\varphi }\left( \varphi (v_{2t}(\tau ,x))\right) dxd\tau d\sigma
\end{equation*}%
\begin{equation}
\leq T\underset{0}{\overset{T}{\int }}\underset{\left\{ x:x\in Q,\left\vert
v_{2t}(\tau ,x)\right\vert >M\right\} }{\int }\varphi (v_{2t}(\tau
,x))v_{2t}(\tau ,x)dxd\tau ,  \tag{A.22}
\end{equation}%
and
\begin{equation*}
\underset{h\searrow 0}{\lim \sup }\underset{s}{\overset{t}{\int }}\underset{s%
}{\overset{\sigma }{\int }}\underset{\left\{ x:x\in Q,\left\vert v_{2t}(\tau
,x)\right\vert >M\right\} }{\int }\underset{-1}{\overset{1}{\int }}%
M_{\varphi }(-v_{2t}(\tau +\mu h,x))d\mu dxd\tau d\sigma
\end{equation*}%
\begin{equation*}
\leq 2\underset{s}{\overset{t}{\int }}\underset{s}{\overset{\sigma }{\int }}%
\underset{\left\{ x:x\in Q,\left\vert v_{2t}(\tau ,x)\right\vert >M\right\} }%
{\int }M_{\varphi }(-v_{2t}(\tau ,x))dxd\tau d\sigma +
\end{equation*}%
\begin{equation*}
+\underset{h\searrow 0}{\lim \sup }\underset{-1}{\overset{1}{\int }}\underset%
{s}{\overset{t}{\int }}\underset{s}{\overset{\sigma }{\int }}\underset{Q}{%
\int }\left\vert M_{\varphi }(-v_{2t}(\tau +\mu h,x))-M_{\varphi
}(-v_{2t}(\tau ,x))\right\vert dxd\tau d\sigma d\mu
\end{equation*}%
\begin{equation*}
\leq 2\underset{s}{\overset{t}{\int }}\underset{s}{\overset{\sigma }{\int }}%
\underset{\left\{ x:x\in Q,\left\vert v_{2t}(\tau ,x)\right\vert >M\right\} }%
{\int }\left\vert \varphi (-v_{2t}(\tau ,x))\right\vert \left\vert
v_{2t}(\tau ,x)\right\vert dxd\tau d\sigma
\end{equation*}%
\begin{equation}
+\underset{h\searrow 0}{\lim \sup }\underset{-1}{\overset{1}{\int }}\underset%
{s}{\overset{t}{\int }}\underset{s}{\overset{\sigma }{\int }}\underset{Q}{%
\int }\left\vert M_{\varphi }(-v_{2t}(\tau +\mu h,x))-M_{\varphi
}(-v_{2t}(\tau ,x))\right\vert dxd\tau d\sigma d\mu .  \tag{A.23}
\end{equation}%
Now, to pass to the limit under last the integral, we apply Lebesgue's
convergence theorem. Since $v_{2t}\in L^{1}(0,T;L^{1}(Q))$, we have (see
\cite[Remark 3.2]{22})
\begin{equation*}
\underset{h\searrow 0}{\lim }\underset{s}{\overset{\sigma }{\int }}\underset{%
Q}{\int }\left\vert v_{2t}(\tau +\mu h,x)-v_{2t}(\tau ,x)\right\vert dxd\tau
=0,\text{ \ }\forall \lbrack s,\sigma ]\subset (0,T),
\end{equation*}%
which yields
\begin{equation*}
v_{2t}(\cdot +h\mu ,\cdot )\rightarrow v_{2t}(\cdot ,\cdot )\text{ in
measure as }h\rightarrow 0.
\end{equation*}%
Also, for $\mu \in \lbrack -1,1]$ and sufficiently small $h>0$, it is easy
to see that%
\begin{equation*}
\underset{s}{\overset{\sigma }{\int }}\underset{Q}{\int }\left\vert
M_{\varphi }(-v_{2t}(\tau +\mu h,x))-M_{\varphi }(-v_{2t}(\tau
,x))\right\vert dxd\tau
\end{equation*}%
\begin{equation*}
\leq \underset{s+\mu h}{\overset{\sigma +\mu h}{\int }}\underset{Q}{\int }%
\left\vert \varphi (-v_{2t}(\tau ,x))\right\vert \left\vert v_{2t}(\tau
,x)\right\vert dxd\tau +\underset{s}{\overset{\sigma }{\int }}\underset{Q}{%
\int }\left\vert \varphi (-v_{2t}(\tau ,x))\right\vert \left\vert
v_{2t}(\tau ,x)\right\vert dxd\tau
\end{equation*}%
\begin{equation*}
\leq 2\underset{0}{\overset{T}{\int }}\underset{Q}{\int }\left\vert \varphi
(-v_{2t}(\tau ,x))\right\vert \left\vert v_{2t}(\tau ,x)\right\vert dxd\tau .
\end{equation*}%
Hence, by the Lebesgue's convergence theorem,%
\begin{equation}
\underset{h\searrow 0}{\lim \sup }\underset{-1}{\overset{1}{\int }}\underset{%
s}{\overset{t}{\int }}\underset{s}{\overset{\sigma }{\int }}\underset{Q}{%
\int }\left\vert M_{\varphi }(-v_{2t}(\tau +\mu h,x))-M_{\varphi
}(-v_{2t}(\tau ,x))\right\vert dxd\tau d\sigma d\mu =0.  \tag{A.24}
\end{equation}%
Taking into account (A.22)-(A.24) in (A.21), we get
\begin{equation*}
I_{1}^{M}(s,t)\geq -T\underset{0}{\overset{T}{\int }}\underset{\left\{
x:x\in Q,\left\vert v_{2t}(\tau ,x)\right\vert >M\right\} }{\int }\varphi
(v_{2t}(\tau ,x))v_{2t}(\tau ,x)dxd\tau
\end{equation*}%
\begin{equation*}
-T\underset{0}{\overset{T}{\int }}\underset{\left\{ x:x\in Q,\left\vert
v_{2t}(\tau ,x)\right\vert >M\right\} }{\int }\left\vert \varphi
(-v_{2t}(\tau ,x))\right\vert \left\vert v_{2t}(\tau ,x)\right\vert dxd\tau ,%
\text{ \ }\forall \lbrack s,t]\subset (0,T),
\end{equation*}%
and consequently%
\begin{equation*}
\underset{M\rightarrow \infty }{\lim \inf }I_{1}^{M}(s,t)\geq 0,\text{ \ \ \
}\forall \lbrack s,t]\subset (0,T).
\end{equation*}%
By the same way, one can show that%
\begin{equation*}
\underset{M\rightarrow \infty }{\lim \inf }I_{i}^{M}(s,t)\geq 0,\text{ \ \ }%
\forall \lbrack s,t]\subset (0,T),\text{ }i=2,3,4,
\end{equation*}%
which, together with (A.20), gives us (A.17).
\end{proof}

\begin{lemma}
Let $Q\subset R^{N}$ be a measurable set of finite measure and $\varphi $ an
increasing continuous function such that $\varphi (0)=0$. Also assume that $%
w\in L^{\infty }(0,T;L^{\infty }(Q))$ and $w_{t}\in L^{1}(0,T;L^{1}(Q))$. If
\begin{equation*}
\underset{0}{\overset{T}{\int }}\underset{Q}{\int }\left( \left\vert \varphi
(w_{t}(t,x))\right\vert +\left\vert \varphi (-w_{t}(t,x))\right\vert \right)
\left\vert w_{t}(t,x)\right\vert dxdt<\infty \text{,}
\end{equation*}%
then%
\begin{equation*}
\underset{h\rightarrow 0}{\lim \inf }\text{ }\underset{s\leq \sigma \leq t}{%
\max }\underset{\sigma }{\overset{\sigma +\left\vert h\right\vert }{\int }}%
\underset{Q}{\int }\left\vert \varphi (w_{t}(\tau ,x))\frac{w(\tau
+h,x)-w(\tau ,x)}{h}\right\vert dxd\tau =0\text{,}
\end{equation*}%
for every $[s,t]\subset (0,T)$.
\end{lemma}

\begin{proof}
By using techniques of the previous lemma, we find%
\begin{equation*}
\underset{\sigma }{\overset{\sigma +\left\vert h\right\vert }{\int }}%
\underset{Q}{\int }\left\vert \varphi (w_{t}(\tau ,x))\frac{w(\tau
+h,x)-w(\tau ,x)}{h}\right\vert dxd\tau
\end{equation*}%
\begin{equation*}
=\underset{\sigma }{\overset{\sigma +\left\vert h\right\vert }{\int }}%
\underset{\left\{ x:x\in Q,\text{ }w_{t}(\tau ,x)\geq 0\right\} }{\int }%
\varphi (w_{t}(\tau ,x))\left\vert \frac{w(\tau +h,x)-w(\tau ,x)}{h}%
\right\vert dxd\tau
\end{equation*}%
\begin{equation*}
-\underset{\sigma }{\overset{\sigma +\left\vert h\right\vert }{\int }}%
\underset{\left\{ x:x\in Q,\text{ }w_{t}(\tau ,x)<0\right\} }{\int }\varphi
(w_{t}(\tau ,x))\left\vert \frac{w(\tau +h,x)-w(\tau ,x)}{h}\right\vert
dxd\tau
\end{equation*}%
\begin{equation*}
\leq \underset{\sigma }{\overset{\sigma +\left\vert h\right\vert }{\int }}%
\underset{\left\{ x:x\in Q,\text{ }w_{t}(\tau ,x)\geq 0\right\} }{\int }%
N_{\varphi }(\varphi (w_{t}(\tau ,x)))dxd\tau
\end{equation*}%
\begin{equation*}
+\underset{\sigma }{\overset{\sigma +\left\vert h\right\vert }{\int }}%
\underset{\left\{ x:x\in Q,\text{ }w_{t}(\tau ,x)\geq 0\right\} }{\int }%
M_{\varphi }\left( \left\vert \frac{w(\tau +h,x)-w(\tau ,x)}{h}\right\vert
\right) dxd\tau
\end{equation*}%
\begin{equation*}
+\underset{\sigma }{\overset{\sigma +\left\vert h\right\vert }{\int }}%
\underset{\left\{ x:x\in Q,\text{ }w_{t}(\tau ,x)<0\right\} }{\int }%
N_{\varphi }(\varphi (w_{t}(\tau ,x)))dxd\tau
\end{equation*}%
\begin{equation*}
+\underset{\sigma }{\overset{\sigma +\left\vert h\right\vert }{\int }}%
\underset{\left\{ x:x\in Q,\text{ }w_{t}(\tau ,x)<0\right\} }{\int }%
M_{\varphi }\left( -\left\vert \frac{w(\tau +h,x)-w(\tau ,x)}{h}\right\vert
\right) dxd\tau
\end{equation*}%
\begin{equation*}
\leq \underset{\sigma }{\overset{\sigma +\left\vert h\right\vert }{\int }}%
\underset{Q}{\int }N_{\varphi }(\varphi (w_{t}(\tau ,x)))dxd\tau +\underset{%
\sigma }{\overset{\sigma +\left\vert h\right\vert }{\int }}\underset{Q}{\int
}\underset{0}{\overset{1}{\int }}M_{\varphi }(\left\vert w_{t}(\tau +h\mu
,x)\right\vert )d\mu dxd\tau
\end{equation*}%
\begin{equation*}
+\underset{\sigma }{\overset{\sigma +\left\vert h\right\vert }{\int }}%
\underset{Q}{\int }\underset{0}{\overset{1}{\int }}M_{\varphi }(-\left\vert
w_{t}(\tau +h\mu ,x)\right\vert )d\mu dxd\tau \leq \underset{\sigma }{%
\overset{\sigma +\left\vert h\right\vert }{\int }}\underset{Q}{\int }\varphi
(w_{t}(\tau ,x)))w_{t}(\tau ,x)dxd\tau
\end{equation*}%
\begin{equation*}
+\underset{0}{\overset{1}{\int }}\underset{\sigma +h\mu }{\overset{\sigma
+h\mu +\left\vert h\right\vert }{\int }}\underset{Q}{\int }\left( \varphi
(\left\vert w_{t}(\tau ,x)\right\vert )-\varphi (-\left\vert w_{t}(\tau
,x)\right\vert )\right) \left\vert w_{t}(\tau ,x)\right\vert dxd\tau d\mu .
\end{equation*}%
Since, by the conditions on $\varphi $,
\begin{equation*}
\varphi (\left\vert w_{t}(\tau ,x)\right\vert )-\varphi (-\left\vert
w_{t}(\tau ,x)\right\vert )=\left\vert \varphi (w_{t}(\tau ,x))\right\vert
+\left\vert \varphi (-w_{t}(\tau ,x))\right\vert
\end{equation*}%
from the above inequality, it follows that%
\begin{equation*}
\underset{\sigma }{\overset{\sigma +\left\vert h\right\vert }{\int }}%
\underset{Q}{\int }\left\vert \varphi (w_{t}(\tau ,x))\frac{w(\tau
+h,x)-w(\tau ,x)}{h}\right\vert dxd\tau \leq \underset{\sigma }{\overset{%
\sigma +h}{\int }}\underset{Q}{\int }\varphi (w_{t}(\tau ,x)))w_{t}(\tau
,x)dxd\tau
\end{equation*}%
\begin{equation*}
+\underset{0}{\overset{1}{\int }}\underset{\sigma +h\mu }{\overset{\sigma
+h\mu +\left\vert h\right\vert }{\int }}\underset{Q}{\int }\left( \left\vert
\varphi (w_{t}(\tau ,x))\right\vert +\left\vert \varphi (-w_{t}(\tau
,x))\right\vert \right) \left\vert w_{t}(\tau ,x)\right\vert dxd\tau d\mu .
\end{equation*}%
By the absolutely continuity property of the Lebesgue integral, we have
\begin{equation*}
\underset{\sigma }{\overset{\sigma +\left\vert h\right\vert }{\int }}%
\underset{Q}{\int }\varphi (w_{t}(\tau ,x)))w_{t}(\tau ,x)dxd\tau
\rightarrow 0\text{ as }h\rightarrow 0
\end{equation*}%
and
\begin{equation*}
\underset{\sigma +h\mu }{\overset{\sigma +h\mu +\left\vert h\right\vert }{%
\int }}\underset{Q}{\int }\left( \left\vert \varphi (w_{t}(\tau
,x))\right\vert +\left\vert \varphi (-w_{t}(\tau ,x))\right\vert \right)
\left\vert w_{t}(\tau ,x)\right\vert dxd\tau \rightarrow 0\text{ as }%
h\rightarrow 0\text{,}
\end{equation*}%
uniformly with respect to $\sigma \in \lbrack s,t]$ and $\mu \in \lbrack
0,1] $. These approximations, together with the last inequality, complete
the proof.
\end{proof}

\bigskip \textbf{Acknowledgments.} The author is grateful to the referee for
helpful suggestions. 

\end{document}